\documentclass[12pt]{amsart}

\usepackage{amssymb, amsmath, amsfonts, amscd}
\usepackage{enumerate}
\usepackage[raiselinks=false,colorlinks=true,citecolor=blue,urlcolor=blue,
linkcolor=blue,bookmarksopen=true,pdftex]{hyperref}
\usepackage[dvipsnames]{xcolor}
\usepackage{amsmath,tikz-cd}

\usepackage[english]{babel}
\usepackage[utf8]{inputenc}
\usepackage{amsmath,amsfonts}
\usepackage{graphicx}
\usepackage{enumitem}
\usepackage{comment}
\usepackage{setspace}
\usepackage{hyperref}
\usepackage{xcolor}

\usepackage{enumerate}
\usepackage[mathscr]{eucal}
\newtheorem{theorem}{Theorem}[section]
\newtheorem{proposition}[theorem]{Proposition}
\newtheorem{lemma}[theorem]{Lemma}
\newtheorem{remark}[theorem]{Remark}
\newtheorem{definition}[theorem]{Definition}
\newtheorem{example}[theorem]{Example}

\newcommand{\bz}{\mathbb{Z}}

\newcommand{\bq}{\mathbb{Q}}

\newcommand{\fix}{\mathrm{Fix}}

\newtheorem{innercustomgeneric}{\customgenericname}
\providecommand{\customgenericname}{}
\newcommand{\newcustomtheorem}[2]{%
  \newenvironment{#1}[1]
  {%
   \renewcommand\customgenericname{#2}%
   \renewcommand\theinnercustomgeneric{##1}%
   \innercustomgeneric
  }
  {\endinnercustomgeneric}
}
\newcustomtheorem{customthm}{Theorem}
\newcustomtheorem{customprop}{Proposition}
\newcustomtheorem{customlemma}{Lemma}
\newcustomtheorem{customremark}{Remark}

\begin{document}
\baselineskip=15.5pt

\title[Cohomology and K-theory of generalized Dold manifolds]{Cohomology and K-theory of generalized Dold manifolds fibred by complex flag manifolds} 
\author[M. Mandal]{Manas Mandal} 

\address{The Institute of Mathematical Sciences, (HBNI), Chennai 600113.} 

\email{manasm@imsc.res.in}

\author[P. Sankaran]{Parameswaran Sankaran
}
\address{Chennai Mathematical Institute, Kelambakkam 603103.}
\email{sankaran@cmi.ac.in}
\thispagestyle{empty}
\date{}

\begin{abstract}
Let $\nu=(n_1,\ldots, n_s), s\ge 2,$ be a sequence of positive integers and let $n=\sum_{1\le j\le s}n_j$.  Let $\mathbb CG(\nu)=U(n)/(U(n_1)\times \cdots\times U(n_s))$ be the complex flag manifold.  Denote by $P(m,\nu)=P(\mathbb S^m,\mathbb CG(\nu))$ the generalized Dold manifold 
$\mathbb S^m\times \mathbb CG(\nu)/\langle \theta\rangle $ where 
$\theta=\alpha\times \sigma$ with $\alpha:\mathbb S^m\to \mathbb S^m$ 
being the antipodal map and $\sigma:\mathbb CG(\nu)\to \mathbb CG(\nu)$, the 
complex conjugation.  
The manifold $P(m,\nu)$ has the structure of a smooth $\mathbb CG(\nu)$-bundle over the real projective space $\mathbb RP^m.$  
We determine the additive structure of $H^*(P(m,\nu);R)$ when $R=\mathbb Z$ and its ring structure when $R$ is a commutative ring in which $2$ is invertible.  As an application, we determine the additive  
structure of $K(P(m,\nu))$ almost completely and also obtain 
partial results on its ring structure. The results for the singular  
homology are obtained for generalized Dold spaces $P(S,X)=S\times X/\langle \theta\rangle$, 
where $\theta=\alpha\times \sigma$, $\alpha:S\to S$ is a fixed point free involution 
and $\sigma:X\to X$ is an involution with $\mathrm{Fix}(\sigma)\ne \emptyset,$  
for a much wider class of spaces $S$ and $X$.   
\end{abstract}

\maketitle

\section{Introduction} \label{intro}
Let $\alpha:S\to S$ be a fixed point free involution and let $\sigma:X\to X$ be an involution with non-empty fixed point set.  Denote by $\theta$ the involution $\alpha\times \sigma$ on $S\times X.$
 Then $P(S,X)=S\times X/\langle\theta\rangle$ is a {\it generalized Dold space} considered in \cite{mandal-sankaran}. (See also \cite{nath-sankaran}, \cite{sarkar-zvengrowski}.) 
The special case when $S$ is the sphere $\mathbb S^m$ with antipodal involution and 
$X=\mathbb CP^n$ where $\sigma$ is given by complex conjugation on $\mathbb C^{n+1},$
corresponds to the classical Dold manifold \cite{dold}, denoted $P(m,n)$.

In \cite{mandal-sankaran}, we described the $\mathbb Z_2$-cohomology algebra 
of a generalized Dold spaces $P(S,X)$ for a family of spaces which included, among others, the case $P(\mathbb S^n,\mathbb CG(\nu))$ with antipodal involution on the sphere $\mathbb S^m$ and the complex conjugation on the complex flag manifold $\mathbb CG(\nu)$
of type $\nu=(n_1,\ldots, n_s)$.
(Thus $\sigma$ is induced by the complex conjugation on $\mathbb C^n$, where $\nu=(n_1,\ldots, n_s)$.) 

Our goal in this paper is to study the {\it integral} (co)homology and the complex 
$K$-theory of $P(\mathbb S^m,\mathbb CG(\nu)).$ 
 We denote $P(\mathbb S^m,\mathbb CG(\nu))$ by $P(m,\nu)$. 
We determine the integral cohomology {\it groups} of $P(m,\nu)$.  In fact 
the additive structure of $H^*(P(S,X);\mathbb Z)$ is obtained for a much wider 
class of spaces. 
We obtain a presentation of the cohomology ring of $P(m,\nu)$ with coefficients in a commutative ring $R,$ when $2$ is invertible in $R$.  
The complex $K$-ring of the classical Dold manifold $P(m,n)$ was determined by Fujii in
\cite{fujii-66}, \cite{fujii-69}.  
Note that if $\omega$ is a complex vector bundle over a 
finite CW complex $B$, 
one has the associated  $\mathbb CG(\nu)$-bundle 
$\mathbb CG(\omega;\nu)\to B.$  In this case, the complex $K$-ring of 
$\mathbb CG(\omega;\nu)$ is well-known. 
See \cite[\S3,Chapter IV]{karoubi} and also \cite{atiyah}.  However, it is easily seen that $P(m,\nu)$ is 
not homeomorphic to $\mathbb CG(\omega;\nu)$ for any complex vector bundle $\omega$ over 
$\mathbb RP^m.$  Indeed, the flag manifold bundle $\mathbb CG(\nu)\hookrightarrow\mathbb CG(\omega;\nu)\to \mathbb RP^m$ admits a cohomology extension of the fibre over $\mathbb Z$ and so, by the Leray-Hirsch theorem \cite[\S7, Chapter 5]{spanier}, $H^*(\mathbb CG(\omega;\nu); \mathbb Z)$ is a free $H^*(\mathbb RP^m;\mathbb Z)$-module of rank equal to $\textrm{rank}_{\mathbb Z}(H^*(\mathbb CG(\nu);\mathbb Z))$. This implies, in particular, that the rank of $H^2(\mathbb CG(\omega;\nu);\mathbb Z)$
is positive.  However, $H^2(P(m,\nu);\mathbb Z)$ has rank $0$---see Theorem \ref{Integral-cohomolgy-groups-of-P(m,n,k)}.
Also, it should be noted that when $m\ge 1, ~P(m,\nu)$ is not a homogeneous space 
of $ Spin(m)\times U(n)$, although $\mathbb S^m\times \mathbb CG(\nu)$ is.   
As such the computation of the $K$-theory of $P(m,\nu)$ is a nontrivial 
and interesting problem.  

Our approach to the computation of $K(P(m,\nu))$ is applicable for a larger class of manifolds, such as generalized 
Dold spaces fibred by a non-singular projective toric variety over $\mathbb RP^m,$ 
but we confine our attention to $P(m,\nu)$.

We shall now state the main results of this paper in their weaker form so as to keep notations to a minimum.   

Denote by $\nu_o$ the number of odd numbers in the sequence $\nu$ and let $\lfloor \nu/2\rfloor:=(n'_1,\ldots, n'_s)$ where $n_j':=\lfloor n_j/2\rfloor.$  As usual, 
${n\choose \nu}$ denotes the multinomial coefficient $n!/(n_1!\cdots n_s!)$.  We set 
$H^{\textrm{ev}}(B;\mathbb Z):=\bigoplus_{q\ge 0} H^{2q}(B;\mathbb Z)$; $H^{\textrm{odd}}(B;\mathbb Z)$ is defined analogously.

\begin{customthm}{A}[See Theorem \ref{betti-torsion-coeffecients}]\label{main-c}
Let $m\ge 1,\nu=(n_1,\ldots,n_s), s\ge 2, n:=|\nu|=\sum_{1\le j\le s} n_j.$  Then \\
    \[H^{\mathrm{ev}}(P(m,\nu);\mathbb Z)\cong \mathbb Z^{b_e}\oplus \mathbb Z_2^{b'_e},~~
    H^{\mathrm{odd}}(P(m,\nu);\mathbb Z)\cong \mathbb Z^{b_o}\oplus\mathbb Z_2^{b'_o}\] where 
    \[b_e={n\choose \nu},~b_o=0, \textrm{if~} m\equiv 0\pmod 2,\]
    \[ b_e=b_o=\ell_e=\begin{cases}
      {n\choose \nu}/2, & \textrm{~if~} \nu_o\ge 2\\
      {n\choose \nu}+{\lfloor n/2\rfloor \choose \lfloor \nu/2\rfloor }& \textrm{~if~} \nu_o\le 1,\\
    \end{cases} ~\textrm{if~}m\equiv 1\pmod 2,
    \]
    and, \[b'_e=\lfloor m/2\rfloor\cdot \ell_e, b'_o=\lfloor (m-1)/2\rfloor \cdot \ell_e.\]
\end{customthm}

A description of $H^*(P(m,\nu);R)$ as a quotient of a polynomial ring is given in 
Theorem \ref{cohomologyringofPmnu}, 
when $R$ is any commutative ring in which 
$2$ is invertible. 

Next, we state our result on the additive structure of $K^*(P(m,\nu)).$
\begin{customthm}{B}[See Theorem \ref{k-groups-pmnu}] \label{main-k}
    With notations as in Theorem \ref{main-c},\\ 
  (i)  $K^0(P(m,\nu))\cong \mathbb Z^{b_e}\oplus \mathbb Z_2^{\lfloor m/2\rfloor}\oplus A_0$ where $o(A_0)=2^t$ for some $t$, $0\le t\le b'_e-\lfloor m/2\rfloor.$\\
  (ii) $K^1(P(m,\nu))\cong \mathbb Z^{b_o}\oplus A_1$ where $o(A_1)=2^t$ for some $t$, $0\le t\le b'_o.$

\end{customthm}

We describe in \S4, a subring $\mathcal K^0$ of $K^0(P(m,\nu))$ generated by the classes of certain vector bundles and show that as an abelian group, $K^0(P(m,\nu))/\mathcal K^0$ is finite. See Theorem \ref{kpmnu-cofinite-ring}.

{\it Notation:} Let $f:X\to Y$ be a continuous map. 
We shall use the same symbol $f^!$ for pull-back of a vector bundle 
over $Y$ to $X$ and also for the induced homomorphism $K^*(Y)\to K^*(X)$ so as to avoid 
clash with the notation $f^*$ for the induced map in cohomology.

\section{Preliminaries}
In this section, we shall obtain some basic results concerning the topology 
of generalized Dold spaces $P(S,X)$ which will be needed for our purposes. 
Some results (and notations) from \cite{nath-sankaran} and \cite{mandal-sankaran} 
will also be recalled.

Let $(S,\alpha),(X,\sigma)$ be as in the Introduction. Then $\theta=\alpha\times \sigma:S\times X\to S\times X$ is a fixed point free involution. 
Set $Y:=S/\!\!\sim_\alpha$ and $X_\mathbb R:=\fix(\sigma)$.
Note that one has a fibre bundle with projection $p:P(S,X)\to Y$, defined as $[u,x]
\mapsto [u]\in Y$, with fibre space $X$.   
A useful fact 
is that we have an inclusion $Y\times X_\mathbb R\hookrightarrow P(S,X)$ defined as 
$([u],x)\mapsto [u,x]$ and the projection $p|_{Y\times X_\mathbb R}$ 
is the first projection 
$pr_1: Y\times Fix(\sigma)\to Y$.  In particular, for 
any $x\in X_\mathbb R$, we have a cross-section $Y\to P(S,X)$ defined as $s_x([u])=[u,x].$

We generalize the construction of the vector bundle $\hat{\omega}$ over $P(S,X)$ for a $\sigma$ conjugate vector bundle $\omega$ over $X$. 

Suppose that $\omega$ is a real vector bundle 
over a space $B$ with projection $p_\omega:E(\omega)\to B$.  
Let $f:B\to B$ be a fixed point free involution that is covered by an involutive bundle automorphism $\hat {f}:E(\omega )\to E(\omega).$   We obtain a vector bundle $\hat{\omega}$ over $\bar B:=B/\!\!\sim_f$ with total space $E(\hat{\omega})=E(\omega)/\!\sim_{\hat{f}}$.  The bundle projection $E(\hat{\omega})\to \bar B$ sends $[v]=\{v,\hat{f}(v)\}$ to $[b]$ for all $v\in E_{b}(\omega).$  If $\omega$ is a complex vector bundle and $\hat{f}$ is a complex vector bundle morphism, then $\hat{\omega}$ is a 
complex vector bundle over $\bar B$.  We emphasize that the isomorphism class of 
$\hat \omega$ depends not only on $\omega$, but also on $\hat f$.

Recall from \cite{nath-sankaran} that a $\sigma$-conjugate vector bundle $\eta$ over $(X,\sigma)$ is a (real) vector bundle with a bundle involution $\hat\sigma:E(\eta)\to E(\eta)$ that covers $\sigma:X\to X$.   Such a pair $(\eta,\hat\sigma)$ 
leads to the construction of a real vector bundle $\hat \eta$ over $P(S,X)$ as follows:  The bundle involution $\hat \sigma$ yields a fixed point free involution 
$\hat \theta:S\times E(\eta)\to S\times E(\eta)$ which is also 
a vector bundle morphism of $pr_2^!(\eta)$ that covers $\theta.$  It follows that 
$P(S, E(\eta))\to P(S,X)$ is the projection of a vector bundle, denoted $\hat\eta.$

A complex vector bundle $\omega$ is said  
to be a $\sigma$-{\em conjugate vector bundle} if $\hat \sigma:E(\omega)\to E(\omega)$ is a bundle involution that covers $\sigma$ and is {\em conjugate complex 
linear} on the fibres of $\omega$.  Then the vector bundle $\hat\omega$ on $P(S,X)$ is isomorphic to $\hat{\bar{\omega}}.$  Moreover, its complexification  
$\hat\omega\otimes_\mathbb R\mathbb C$ restricts to $\omega\oplus 
\bar\omega$ on any fibre of the $X$-bundle $p:P(S,X)\to Y$.

The construction of $\hat\omega$  behaves well with respect to Whitney sum, 
tensor products, taking exterior powers, etc., so long as all the bundles involved are  $\sigma$-conjugate (complex) vector bundles, i.e., the respective bundle involutions 
cover the {\em same} $\sigma:X\to X.$  (See \cite{nath-sankaran}.)

\subsection{Half-spin bundles over $\mathbb S^{2r}$}\label{bott-bundle}

Assume that $m=2r$ is even.  An explicit description of a complex vector bundle $\xi$ 
over $\mathbb S^m$ such that $[\xi]-\textrm{rank}(\xi)$ is a generator 
of $\tilde K(S^m)\cong \mathbb Z$ is well-known.   See \cite{atiyah-hirzebruch}, \cite{bott}. 
For the sake of completeness, we give the details.  

Explicitly, $\xi$ may be taken to be associated 
to a half-spin representation of $\textrm{Spin}(m)$ where $\mathbb S^m$ is viewed as the homogeneous space $\textrm{Spin}(m+1)/\textrm{Spin}(m)$. We shall explain this 
construction below, assuming familiarity with the basic properties of the spin groups, as given in \cite{husemoller}. 

We consider $\mathbb R^{m+1}$ as the standard $SO(m+1)$-representation. It becomes 
a representation of $\textrm{Spin}(m+1)$ via the projection $\textrm{Spin}(m+1)\to 
SO(m+1).$
Regard $\textrm{Spin}(m)\subset \textrm{Spin}(m+1)$ is the subgroup that stabilizes 
the vector $e_1\in \mathbb R^{m+1}.$ 
Consider bundles $\xi^+,\xi^-$ 
associated to the 
half-spin (complex) representations $\Delta^+,\Delta^-$ of the spin group $\textrm{Spin}(m)$. Then $[\xi^+]-2^{r-1}$ is a generator of 
$\tilde K(\mathbb S^m)$.  See  \cite{atiyah-hirzebruch}, \cite[Theorem 13.3, Chapter 14]{husemoller}.
(Note that the degrees of the 
representations $\Delta^\pm$ equal $2^{r-1}.$)  It turns out that 
$[\xi^+]+[\xi^{-}]=2^r$ in $\tilde K(\mathbb S^m).$ 
In fact, the representation 
$\Delta^+\oplus \Delta^-$ is the restriction of the spin representation $\Delta$ of 
$\textrm{Spin}(m+1)$.  As the bundle associated to $\Delta|_{\textrm{Spin}(m)}$ 
is trivial, we obtain that $\xi^+\oplus \xi^-=2^r\epsilon_\mathbb C$, whence 
\begin{equation}\label{xiplusminus}
[\xi^+]-2^{r-1}=-([\xi^-]-2^{r-1}).
\end{equation}

\begin{theorem} \label{half-spin-bundles on even spheres} Let $m=2r$ be even.
The ring $K^0(\mathbb S^m)$ is isomorphic to the truncated polynomial ring $\mathbb Z[u_m]/\langle u_m^2\rangle $ under an isomorphism that sends $u_m$ to $[\xi^+]-2^{r-1}.$   In particular, $[\xi^\pm]^2=2^r[\xi^\pm]-2^{2r-2},$ and, $[\xi^+][\xi^-]=2^{2r-2}.$  Moreover, $\alpha^!:K^0(\mathbb S^m)\to K^0(\mathbb S^m)$ sends $[\xi^+]$ to $[\xi^-].$ 
\end{theorem}
\begin{proof}  The description of the ring structure 
and the fact that $u_m$ is represented by $[\xi^+]-2^{r-1},$ are well-known. 
(See \cite{atiyah-hirzebruch}
and \cite{husemoller}.)
The asserted quadratic relation follows from 
$([\xi^+]-2^{r-1})^2=u_m^2=0.$ Since $[\xi^+]+[\xi^-]=2^r$, the same relation 
is satisfied by $[\xi^-].$  Also, $[\xi^+][\xi^-]=[\xi^+](2^r-[\xi^+])=2^{2r-2}.$

It remains to show that $\alpha^!(u_m)=-u_m.$
This follows from the following facts: $H^m(\alpha;\mathbb Q)$ is $-id$, the Chern character $\mathrm{ch}:K(\mathbb S^m)
\to H^*(\mathbb S^m;\mathbb Q)$ is a monomorphism, and $\alpha^*\circ \textrm{ch}=\textrm{ch}\circ \alpha^!$.  
\end{proof}

We have not been able to find a natural and explicit bundle isomorphism $\xi^-\to \xi^+$ which covers the antipodal map $\alpha:\mathbb S^m\to \mathbb S^m$.  For this reason, we work with the pull-back bundle $\alpha^!(\xi^+)$, which 
will suffice for our purposes.

Denote by $\eta^-$ the pull-back bundle $\alpha^!(\xi^+).$   When $m\ge 2$ is even, we have $ 2^{r-1}\ge  m/2$ 
and any two vector bundles of rank $2^{r-1}$ which represent the same class in $K(\mathbb S^m)$ 
are isomorphic, using \cite[Theorem 1.5, Chapter 8]{husemoller}.

For future reference, we have the following remark.
\begin{remark}\label{eta-xi}
    From the above discussion, we have $[\eta^-]=[\xi^-]=2^r-[\xi^+]$ in $K(\mathbb S^m)$ for all $m\equiv 0\pmod 2$. 
\end{remark}

We have a bundle isomorphism $\tilde \alpha^-:E(\eta^-)\subset \mathbb S^m\times E(\xi^+)\to E(\xi^+)$ given by the second projection that covers the antipodal map $\alpha$.   Explicitly, $\tilde \alpha^-:E(\eta^-)\to E(\xi^+)$ is the map that $(v,e)\in E_v(\eta^-)$ to $e\in E_{-v}(\xi^+)$.  Then $\tilde \alpha^-$ is a bundle map that covers $\alpha$.  Similarly $\tilde \alpha^+:E(\xi^+)\to E(\eta^-)$ is the bundle map that sends $e\in E_v(\xi^+)
$ to $(-v,e)\in E_{-v}(\eta^-)$.  Henceforth, we shall identify the fibre $E_v(\eta^-)=\{v\}\times E_{-v}(\xi^+)$ with 
$E_{-v}(\xi^+)$ so that the total space $E(\eta^-)$ is identified with $E(\xi^+).$  We have, for any $v\in \mathbb S^m$ and $e\in E_v(\xi^+)$, $\tilde \alpha^+(\tilde \alpha^-(e))=e$, i.e., $\tilde \alpha^+\circ \tilde \alpha^-=id_{\eta^-}$.  Similarly $\tilde \alpha^-\circ \tilde \alpha^+=id_{\xi^+}$.  

Set $\tilde \xi:=\xi^+\oplus \eta^-$.  A vector in the fibre $E_v(\tilde \xi)$ over $v\in \mathbb S^m$ will be denoted by a triple $(v;e,e'), e\in E_v(\xi^+), e'\in E_{-v}(\xi^+)=E_v(\eta^-), v\in \mathbb S^m.$  
Define 
$\tilde\alpha: E(\xi^+\oplus \eta^-)\to E(\xi^+\oplus \eta^-)$ as $\tilde\alpha (v;e,e')=(-v;
\tilde\alpha^-_v(e'), \tilde\alpha^+_{v}(e))$.   Then $\tilde\alpha $ covers $\alpha$ and 
$\tilde\alpha\circ \tilde \alpha=id.$  

We therefore obtain a complex 
vector bundle $\xi^0$ over $\mathbb RP^m$ where 
$E(\xi^0):=E(\tilde \xi)/\langle \tilde\alpha\rangle$.   
A point of $E_{[v]}(\xi^0)$ is $[v;x,y]=\{ (v;x,y), (-v; y,x)\}$ where $x\in E_v(\xi^+), y\in E_v(\eta^-)=E_{-v}(\xi^+)$.

    Let us consider the bundle maps $\tilde\beta^+:E(\eta^-) \subset \mathbb S^m\times E(\xi^+)\to E(\xi^+)$ defined by $(v,e)\mapsto ie\in E_{-v}(\xi^+)$ and similarly $\tilde \beta^-:E(\xi^+)\to E(\eta^-)$ which maps $e\in E_{v}(\xi^+)$ to $(-v,ie)\in E_{-v}(\eta^-)$. These two bundle maps cover $\alpha$. It is easy to see that $\tilde\beta^+\circ\tilde\beta^-=-id_{E(\eta^-)}$ and $\tilde\beta^-\circ\tilde\beta^+=-id_{E(\xi^+)}$. 
    Define $\tilde\beta: E(\tilde\xi)\to E(\tilde \xi)$, $(v;e,e')\mapsto (-v;-\tilde \beta^-(e'),\tilde \beta^+(e) )$ which is a bundle involution covering $\alpha$.
    Now we can form a vector bundles $\xi^1$ over $\mathbb RP^m$ where $E(\xi^1)=E(\tilde \xi)/\langle\tilde\beta\rangle$. A point of $E_{[v]}(\xi^1)$ is $[v;x,y]=\{ (v;x,y), (-v; -iy,ix)\}$ where $x\in E_v(\xi^+), y\in E_v(\eta^-)=E_{-v}(\xi^+)$.
\begin{remark}
    We note that $\xi^0$ and $\xi^1$ are isomorphic. An explicit isomorphism can be defined by $E(\xi^0)\ni[v;x,y]\mapsto [v; -ix,y]\in E(\xi^1).$
 
\end{remark}

\section{Cohomology of certain generalized Dold manifolds}
In \cite{mandal-sankaran}, we determined the singular mod $2$
cohomology of certain families of generalized Dold spaces. 
Here we consider the integral cohomology groups of $P(S ,X).$ 
We assume that $X$ has a CW structure with cells only in even dimensions and 
that each cell is stable by the involution $\sigma.$
Examples of such spaces are complex flag manifolds and complex projective toric varieties 
where the involution is given by complex conjugation. In the case of 
complex flag manifolds the CW structure is given by Schubert varieties (see \cite{borel-lag})
and, in the case of complex projective toric varieties, the required cell structure 
is obtained as orbit closures for the natural complex torus $\cong (\mathbb C^*)^n$ 
action on it.

{\bf Notation:}  Let $B$ be a (locally finite) CW complex.   We shall denote the closed cells of $B$ by $e,e_j,$ etc., and the corresponding open cells by $\mathring {e}, \mathring {e}_j$, etc.   An orientation on a 
(closed) $k$-cell $e$ in $B$ is the choice of a generator of $H_k(e,\partial e;\mathbb Z)\cong 
H_k(\mathbb S^k;\mathbb Z).$ An orientation on $e$ is equivalent to an orientation of 
the open cell $\mathring e\cong\mathbb R^k.$  

\subsection{Homology of $P(S,X)$}
In this section, we shall consider the homology of $P(S,X)$ where $X$ is a connected locally finite CW complex that has only even dimensional cells.  We assume that $S$ is a connected locally finite CW complex.  Further restrictions on $S$ and $X$ will be placed later as and when required. 
Some of the results here are established in the context of a double cover $p:\widetilde B\to B$ with appropriate hypotheses on $\widetilde B$.

\begin{lemma}\label{conjugation-cw}
Let $X$ be a connected locally finite CW complex with only even dimensional cells.  Suppose 
that $\sigma:X\to X$ is an involution that stabilizes each cell in $X$ and that 
$\sigma$ is orientation preserving on 
any $2k$-dimensional cell $e$ if and only if $k$ is even.
Then $\sigma_*$ acts  by $(-1)^k$ on $H_{2k}(X;\mathbb Z).$\hfill $\Box$
\end{lemma}
We omit the proof, which is straightforward.

Suppose that $p:\tilde B\to B$ is a double covering projection where $B$ is a connected locally finite CW complex. We obtain a $G$-equivariant CW structure on $\tilde B$ by lifting the cell structure on $B$ where $G\cong \mathbb Z_2$ is the deck transformation group. Denote by $\phi: \tilde B\to \tilde B$ the involution that generates $G$.
For each (closed) cell $e$ of $B$, we have a pair of cells $e^+,e^-$ in $\tilde B$ such that $p(e^\pm)=e$ and $\phi(e^+)=e^-$ (as unoriented cells).

Let $R$ be any commutative ring. (It is always assumed that $0\ne 1\in R.$)
We denote by $(C_*(B;R),\partial)$ the cellular chain complex with $R$-coefficients.
Recall that $C_r(B;R)$ is the free $R$-module with basis the set of all closed oriented 
$r$-cells of $B$ modulo the relations $e +e'=0$ 
where $e'$ is the same underlying cell as $e$ but with opposite orientation. 
(We denote by the same symbol $e$  an oriented cell as well as the corresponding element in the 
cellular chain group $C_k(B;R )$.)
It is clear 
that $C_r(B;R)$ is isomorphic to the free $R$-module with basis the set of 
{\em unoriented} $r$-cells of $B$. 

For each oriented cell $e$ in $B$, fix orientations on $e^+$ and $e^-$ so that $p|_{e^\pm}:e^\pm\to e$ is orientation preserving; thus $p_*(e^\pm)=e$. Then $\phi|_{e^+}:e^+\to e^-$ is orientation preserving since $p=p\circ \phi.$

Suppose that $2$ is invertible in $R$.  
We set $\varepsilon^+:=(e^++e^-)/2$ and $\varepsilon ^-:=(e^+-e^-)/2$ for each closed cell $e$ 
of $B$.  Then $C_r(\tilde B;R)=C_r^+(\tilde B;R)\oplus C_r^-(\tilde B;R)$ as a direct sum 
of $R$-modules where $C^+_r(\tilde B;R)$ is the free $R$-module with basis 
$\{\varepsilon_i^+\}_i$ as $e_i$ varies over the set of $r$-cells of $B$.  
Similarly $\{\varepsilon^-_i\}_i$ is an $R$-basis for 
$C_r^-(\tilde B;R)$.  It is readily seen that $C^+_r(\tilde B;R)=\fix(\phi_*)$, the $R$-submodule of $C_r(\tilde B;R)$ element-wise fixed by $\phi_*$ and that $\phi_*$ acts as multiplication by $-1$ on 
$C_r^-(\tilde B;R)$. 
Clearly $p_*$ maps $C_r^+(\tilde B; R)$ isomorphically onto $C_*(B;R)$ and vanishes identically on $C_r^-(\tilde B;R).$  Note that $\partial:C_*(\tilde B;R)\to C_*(\tilde B;R)$ 
maps $C_*^\pm(\tilde B;R)$ to itself since $\partial \circ \phi_*=\phi_*\circ \partial.$  Hence 
$H_r(\tilde B;R) $ breaks up as  a direct sum of two $R$-modules on which $\phi_*$ acts 
by $1$ and $(-1)$ respectively.

\begin{lemma} \label{ker-charFne2}
We keep the above notations. Let $R$ be any commutative ring in which $2$ is invertible.
Then $p_*:H_*(\tilde B;R)\to H_*(B;R)$ maps $\fix(\phi_*)$ isomorphically onto $H_*(B;R).$ Moreover
$p_*[z]=0$ if $\phi_*([z])=-[z]$.
\end{lemma}
\begin{proof}
 Since $p_*=p_*\circ \phi_*$ in homology, if 
$\phi_*([z])=-[z]$, then $p_*[z]=-p_*[z]$. Since 
$2$ is invertible in $R$, we have $p_*([z])=0.$
 
Suppose that $\phi_*[z]=[z]$.  Write 
$z=z^++z^-$ where $z^\pm\in C_r^\pm (\tilde B;R).$
Since $\partial z_i^\pm \in C_r^\pm(\tilde B;R)$ (as observed above), we have $\partial z^\pm=0$ and so $[z]=[z^+]+[z^-]$.  Since $[z]=\phi_*([z])=[z^+]-[z^-]$, we have $2[z^-]=0$ which implies that $[z^-]=0$ as $2$ is a unit in $R$.  So we may (and do) assume that $z=z^+$.  

Suppose that $p_*[z]=0.$ 
Write $p_*(z)=\partial (\sum a_je_j)$ and set $\tilde u:=\sum a_j\varepsilon_j^+\in C_{r+1}^+(\tilde B)$.  Then $p_*(\partial (\tilde u))=\partial p_*(\tilde u)=p_*(z).$
Since $p_*:C_*^+(\tilde B;R)\to C_*(B;R)$ is a monomorphism, we have $z=\partial \tilde u$ and so $[z]=0.$  Thus $p_*|_{\fix(\phi_*)}$ is a monomorphism.

Next suppose that $z=\sum a_je_j$ is a cycle in $B$ (with $a_j\in R$).  Then $p_*(\tilde z)=z$ where 
$\tilde z=\sum a_j \varepsilon_j\in C_r^+(\tilde B;R).$    Since $p_*(\partial \tilde z)
=\partial z=0$ and since $p_*:C_*^+(\tilde B;R)\to C_*(B;R)$ is a monomorphism, 
we see that $\partial \tilde z=0$. Clearly $p_*([\tilde z])=[z]$ and so $p_*$ is an epimorphism.
\end{proof}

As an immediate corollary, we obtain the following.

\begin{proposition} \label{homology-integral-fixed}
Let $[z]\in H_r(\tilde B;\mathbb Z).$
If $p_*([z])=0$ and $\phi_*([z])=[z],$
then $2^m[z]=0$ for some $m\ge 1$.
 If 
$\phi_*([z])=-[z]$, then $2p_*([z])=0$.  \hfill $\Box$
\end{proposition}

\begin{example}{\em 
    Let $B$ be the $2n$-dunce hat, defined as the mapping cone of $f:\mathbb S^1\to \mathbb S^1$ where 
    $f(z)=z^{2n}.$   Assume that $n>1$.   
    Then $\tilde B$ is a two dimensional complex, $\pi_1(\tilde B)\cong \mathbb Z_n\to \mathbb Z_{2n}\cong\pi_1(B)$ 
    is the inclusion map.   We have $H_k(\tilde B;\mathbb Z)=0$ for $k\ge  2$.  
    The deck transformation group induces $id$ on $H_1(\tilde B;\mathbb Z)\cong \mathbb Z_n$ 
    and  $p_*:H_1(\tilde B;\mathbb Z)\to H_1(B;\mathbb Z)$   
    corresponds to inclusion $\mathbb Z_n\hookrightarrow \mathbb Z_{2n}$. We see that $H_*(p;R)$ is an 
    isomorphism if $2$ is a unit in $R$.  
    Although $B$ has a CW complex structure which is perfect mod $2$, 
    taking $n=2^k$, $k\ge 1,$ the $2$-torsion elements in $H_1(B;\mathbb Z)$ need not be of order $2$.
    } \hfill $\Box$
\end{example}

Suppose that $\pi_1(\tilde B)\cong \mathbb Z^n$.   We have an exact sequence 
of groups where $C_2\cong \mathbb Z_2$:
\begin{equation}\label{exactsequence-pi1}
 1\to \pi_1(\tilde B)\to \pi_1(B)\to  C_2\to 1.
\end{equation}
If $\pi_1(B)$ is abelian, then either $\pi_1(B)\cong \mathbb Z^n\oplus \mathbb Z_2$ or $\pi_1(B)\cong \mathbb Z^n$ according as whether the exact sequence splits or not.
Hence $H_1(B;\mathbb Z)\cong \mathbb Z^n\oplus \mathbb Z_2$ or $\mathbb Z^n.$ 

Suppose that $\pi_1(B)$ is not abelian. 
We will use additive notation for $\pi_1(\tilde B)$.   
Conjugation by $y\in \pi_1(B)\backslash \pi_1(\tilde{B})$ defines an automorphism $A:\pi_1(\tilde B)\to \pi_1(\tilde B)$.  Fixing a basis $x_1,\ldots, x_n$, we obtain a matrix 
$A=(a_{ij})\in GL(n,\mathbb Z)$
such that $A^2=I_n$ and $yxy^{-1}=Ax~\forall x\in \mathbb Z^n.$  Since $\pi_1(B)$ is not abelian, $A\ne I_n$.
We assume that $y$ is chosen so that it has order $2$ when the above exact sequence splits.  
The abelianization of $\pi_1(B)$, namely $H_1(B;\mathbb Z),$ has the following presentation 
\[H_1(B,\mathbb Z)=(\mathbb Z x_1 + \cdots +\mathbb Z  x_n + \bz y)/H\] 
where $H$ is generated by the following elements  
$ 2y-\sum_{1\le i\le n}c_ix_i, (A-I_n)x_j, 
1\le j\le n
$,
for some $c\in \mathbb Z^n.$  If $c=0$, then $y$ generates a cyclic subgroup of order $2$. Suppose that $c\ne 0.$  Then the above exact sequence does not split by our hypothesis on $y$, and there exists a $j$ so that $c_j$ is odd. 
We may assume that all the non-zero $c_i$ 
are $1$. Relabeling the $x_j$, we may assume that $c_1=1$.  Now replacing  
$x_1$ by $\sum c_ix_i$ in the basis $x_1,\ldots, x_n$ of $\pi_1(B)$ if necessary,  
we may (and do) assume that $2y=x_1.$  
Note that $yx_1y^{-1}=x_1.$ It follows that 
the image of $y$ is an element of infinite order in $H_1(B;\mathbb Z)$ and 
that $H_1(B;\mathbb Z)$ is generated by at most $n$ elements.  

Set $A_-:=\{x\in \pi_1(\tilde B)\mid Ax=-x\}\subset \pi_1(B).$ 
Its image $\bar A_-$ in $H_1(B;\mathbb Z)$ is isomorphic to $\mathbb Z_2^k$ where $k=\textrm{rank}(A_-).$ 

We claim that $H_1(B;\mathbb Z)/\bar A_-$ is a free abelian group of rank $n-k.$
To see this, let $A_+=\{x\in \pi_1(\tilde B)\mid Ax=x\}\subset \pi_1(B).$ 
Since $A$ is invertible and since $A_++A_-$ has the same rank $n$ as $\pi_1(\tilde{B}),$ it 
suffices to show that $A_+$ maps monomorphically into $H_1(B;\mathbb Z)$.  
This is immediate from the above presentation of $H_1(B;\mathbb Z)$ since $yxy^{-1}=Ax=x.$  We summarise the result in the proposition below.

\begin{proposition} \label{first-homology}
Let $p:\tilde B\to B$ be a double cover where $B$ is connected and locally path connected.  
    Suppose that $\pi_1(\tilde B)$ is free abelian of finite rank, say $n$.   
    Then 
    $H_1(B;\mathbb Z)\cong \mathbb Z^l\oplus \mathbb Z_2^k$ for some $k,l\ge 0 $.  Moreover, 
    $k+l=n$ except when the deck transformation group $\mathbb Z_2$ acts trivially 
    on $\pi_1(\tilde B),$ in which case $l=n, k=1.$ \hfill $\Box$ 
\end{proposition}

We apply Lemma \ref{ker-charFne2} to the special case of the double covering projection 
$\pi: S\times X\to P(S,X)$ satisfying further hypotheses.
Denote by $H^\pm_p(S;R) $ the $R$-submodule of $H_p(S;R)$ on which $\alpha_*$ acts as $\pm1$ and denote its 
dimension by $b^\pm_p$. Thus the $r$-th Betti number $b_r(S)$ of $S$ equals $b_r^++b_r^-$.

\begin{proposition} \label{nooddtorsion} We keep the above notations.  
    Suppose that $Y$ is a connected locally finite CW complex and 
    that $H_*(S;\mathbb Z)$ is free abelian.  Assume that $X$ is a connected locally finite CW complex with cells only in even dimensions. Assume that the involution $\sigma:X\to X$ satisfies the hypothesis of Lemma \ref{conjugation-cw}. Let $R$ be a ring where $2$ is invertible. Then
(i) $\pi_*: H_*(S\times X;R)\to H_*(P(S,X);R)$ is a monomorphism. In fact, 
we have an isomorphism
\begin{equation}\label{homology-psx}
H_r(P(S,X);R)\cong \bigoplus_{p+4t=r}(H_p^+(S;R)\otimes H_{4t}(X;R)
\oplus H^-_{p-2}(S;R)\otimes H_{4t+2}(X;R))
\end{equation}
In particular,  
\begin{equation} \label{bettinumber-psx} \dim H_r(P(S,X);\mathbb Q)=\sum_{p+4q=r} (b_p^+b_{4q}(X) +b_{p-2}^-b_{4q+2}(X)).
\end{equation}\\
(ii)
Suppose that $\pi_1(S)$ is abelian.  Any torsion element of $H_r(P(S,X);\mathbb Z)$ has order a power of $2$. 
\end{proposition}
\begin{proof}
  (i). We shall apply Lemma \ref{ker-charFne2}.  
In view of our hypotheses and the Künneth theorem, we have $H_r(S\times X;\mathbb Z)=\bigoplus_{p+2q=r}H_p(S;\mathbb Z)\otimes H_{2q}(X;\mathbb Z)$. Recall that $\pi=\pi\circ \theta$ where $\theta=\alpha\times \sigma.$  So $\pi_*=\pi_*\circ \theta_*$ and so $\theta_*=\alpha_*\otimes \sigma_*.$ 
Since $\sigma_*$ acts on $H_{2q}(X;\mathbb Z)$ by $(-1)^q$ we obtain that $H_r(P(S,X);R) \cong
\fix(H_r(\theta;R))=\bigoplus_{p+4t=r}(H_p^+(S;R)\otimes H_{4t}(X;R))
\oplus (H^-_{p-2}(S;R)\otimes H_{4t+2}(X;R))$. Taking $R=\mathbb Q$
we obtain the asserted formula for the $r$-th Betti number of $P(S,X)$. 

(ii)
We only need to show that there is no odd torsion element in  $H_r(P(S,X);\mathbb Z)$.
Our hypothesis on $\pi_1(S)$ and the freeness of $H_1(S;\mathbb Z)$ implies that $\pi_1(S)\cong \mathbb Z^k$ for some $k\ge 0$.   If $k=0$, then $\pi_1(P(S,X))\cong \mathbb Z_2.$  In  general, $\pi_1(P(S,X))$ has an index $2$ subgroup $H:=\pi_*(\pi_1(S\times X))\cong 
\mathbb Z^k$.  
Then by Proposition \ref{first-homology}, $H_1(P(S,X);\mathbb Z)\cong \mathbb Z^l\oplus\mathbb Z_2^{t}$ for some $l, t\ge 0$ and so has no odd torsion.  

Since $H_*(S\times X;\mathbb Z)$ has no torsion and since $\pi_*$ maps 
$\fix (H_*(\theta;\mathbb Q))$ isomorphically onto $H_*(P(S,X);\mathbb Q)$, $\pi$ induces 
a monomorphism $\fix(H_*(\theta;\mathbb Z))\to H_*(P(S,X);\mathbb Z)$.  
Suppose that $H_r(P(S,X);\mathbb Z)$ has an element of order an odd prime $p$. We assume that $r$ is the least positive integer for which this happens.  So $r>1$ and 
$\textrm{Tor}(H_{r-1}(P(S,X); \mathbb Z),\mathbb Z_p)=0.$
By the universal coefficient theorem, 
$H_r(P(S,X);\mathbb Z_p)\cong H_r(P(S,X);\mathbb Z)\otimes \mathbb Z_p$. 
Since the $r$-th Betti number of $P(S,X)$ equals $\dim H_r(P(S,X);\mathbb Z_p)$ by (i),  
every non-zero element in $H_r(P(S,X);\mathbb Z_p)$ has to be the reduction mod $p$ of 
an element of $H_r(P(S,X);\mathbb Z)$ of {\em infinite} order. This implies that 
$H_r(P(S,X);\mathbb Z)$ has no element of order $p,$ contrary to our choice of $r$.
Hence the proof.
\end{proof}

\begin{remark}\label{orientability}
(i) The cohomology version of the above proposition is valid and is equivalent to it by the universal coefficient theorem. \\
(ii) Suppose that $S,X$ as in the above theorem are compact connected orientable manifolds (without boundary). Suppose that $\dim S=m, \dim X=2d.$  Then $\sigma$ is orientation preserving if and only if $d$ is even.  
It follows that $P(S,X)$ is orientable if and only if $\alpha$ is orientation preserving and $d$ is even, or, $\alpha$
orientation reversing and $d$ is odd. This follows from computation of $H_{m+2d}(P(S,X);\mathbb Q)$.  This is also seen from the fact that $\theta$ is orientation 
preserving if and only if both $\alpha,\sigma$ are simultaneously either orientation 
preserving or reversing. 
\end{remark}

\subsection{The generalized Dold manifold $P(m,\nu)$}  
We shall now consider the specific classes of manifolds $P(S,X)$ where \\ 
(i) $S=\mathbb S^m$, $\alpha:\mathbb S^m\to \mathbb S^m$ is the antipodal map, and, \\
(ii) $X$ is a complex flag manifold 
and $\sigma:X\to X$ is the complex conjugation (induced by 
the complex conjugation $~\bar{}:\mathbb C^n\to \mathbb C^n$).

Recall that $\mathbb CG_{n,k}\cong U(n)/(U(k)\times U(n-k))$ is the space all $k$-dimensional complex vector subspaces of $\mathbb C^n$.  Let $\sigma:\mathbb CG_{n,k}\to \mathbb CG_{n,k}$ be the involution that 
sends $L\in \mathbb CG_{n,k}$ to $\bar L=\{\bar v\in \mathbb C^n\mid v\in L\}$. 

We put the standard hermitian inner product on $\mathbb C^n$. 
The complex flag manifold $\mathbb CG(\nu)$ of type $\nu:=(n_1\cdots, n_s)$ is the space of all flags 
$\mathbf {L}:=(L_1,\cdots, L_s)$ where: (i) $L_j$ is a vector subspace of $\mathbb C^n$ of dimension $n=|\nu|:=\sum_{1\le j\le s} n_j$, and, (ii)  
where $L_i\perp L_j$ if $i\ne j$; thus $L_1+\cdots+L_s=\mathbb C^n.$  The complex flag manifold $\mathbb CG_{\nu}$ is 
identified with the homogeneous space $U(n)/U(\nu)$ where $U(\nu):=(U(n_1)\times \cdots \times U(n_s))\subset U(n)$ is the subgroup of block diagonal matrices with block sizes $n_1,\cdots, n_s$. 
When $s=2,~\mathbb CG(\nu)$ is the Grassmann manifold $\mathbb CG_{n,n_1}$.

\subsection{Schubert cell structure on $\mathbb CG(\nu)$} 
One has the Schubert cell structure for the complex Grassmann manifold $\mathbb CG(\nu)$ (cf. \cite[Chapter 6]{milnor-stasheff}) which has a generalization to the case of complex flag manifolds given by the Bruhat decomposition (see \cite{brion}, \cite{borel}). The closed cells are the Schubert varieties in $\mathbb CG(\nu)$ are in bijection with the set of $T$-fixed points of $\mathbb CG(\nu)$ where $T\subset U(n)$ is the diagonal subgroup, which is a maximal torus of $U(n)$.
These $T$-fixed points are in bijection with the coset space $S_n/S_\nu$ where $S_\nu:=S_{n_1}\times \cdots \times S_{n_s}$.  Here $S_n$ is the permutation group on $[n]=\{1,2, \cdots, n\}$. In the case of $\mathbb CG_{n,k}$,  $S_n/(S_k\times S_{n-k})$ is identified with the set 
$I(n;k)$ of all strictly increasing sequences ${\bf i}=(i_1,i_2,\ldots, i_k), 1\le i_p\le n.$  The Schubert variety  
$X({\bf i})\subset \mathbb CG_{n,k}$ is defined as 
\[ X(\mathbf i):=\{L\in \mathbb CG_{n,k}\mid \dim_\mathbb C(L\cap \mathbb C^{i_p})\ge p~\forall p\le k\}.
\]
The corresponding $T$-fixed point is $E_{\mathbf i},$ the $\mathbb C$-span of the standard basis vectors $e_{i_p}, 1\le p\le k.$  $X(\mathbf i)$ has the structure of a complex projective variety.
One has the following formula for the dimension of $X(\mathbf i)$ as a complex variety: 
\[\dim_\mathbb CX(\mathbf i)=
\ell(\mathbf i):=\sum_{1\le p\le k}(i_p-p).\]   
The corresponding (open) Schubert cell is 
$\mathring X(\mathbf i)=\{L\in \mathbb CG_{n,k}\mid \dim (L\cap \mathbb C^{i_p})=p,\dim (L\cap \mathbb C^{i_p-1})=p-1 ~\forall p\}\cong \mathbb C^{\ell(\mathbf i)}.$ 
We shall denote by $I_q(n;k)\subset I(n;k)$ the set of all $\mathbf i\in I(n;k)$ with $\ell (\mathbf i)=q.$  

In the case of $\mathbb CG(\nu)$, we identify $S_n/S_\nu$ with the subset $I(\nu)\subset S_n$
defined as follows:  $I(\nu)$ consists of all permutations $\mathbf i=(i_1,\cdots, i_n)=(\mathbf i_1,\ldots, \mathbf i_s)$ of $[n]$ 
(in one-line notation) in which 
each `block' $\mathbf i_p$ (of size $n_p$) is in $I(n; n_p), 1\le p\le s,$ so that  
$\mathbf i_{p+1}=i_{m_p+1}<\cdots<i_{m_{p+1}}$ where $m_p:=n_1+\cdots+n_p$ (with $m_0=0$), $0\le p<s$.  
With respect to the generating set $S\subset S_n$ consisting of transpositions $s_i=(i,i+1), 1\le i<n$,  one has the length function $\ell=\ell_S$ on the 
Coxeter group $(S_n,S)$.  
The set $I(\nu)$ is 
the set of minimal length representatives of elements of $S_n/S_\nu$.  We have the formula: 
\[
\ell(\mathbf i)
=|\{(t,r)|1\le t<r\le n; i_r<i_t\}|=\sum_{1\le t<n}|\{r\mid t<r\le n, i_r<i_t\}|.
\]
When $s=2$, this is consistent with the notation 
$\ell(\mathbf i)$ for $\mathbf i\in I(n;k)$.  
The above formula for $\ell(\mathbf i), \mathbf i\in I(\nu),$ for $s\ge 3$ can be 
proved by induction on $s$: If $\mathbf i\in I(\nu)$, write $\nu'=(n_1,\ldots, n_{s-2}, n_{s-1}+n_s)$ and let $\mathbf i'=(\mathbf i_1,\cdots, \mathbf i_{s-2}, \mathbf i'_{s-1})\in I(\nu')$ where $\mathbf i'_{s-1}\in I(n;n_{s-1}+n_s)$ is the sequence $\mathbf i_{s-1},\mathbf i_s$ rearranged in the increasing order. 
Then, it is easily seen that 
\[\ell(\mathbf i)=\ell(\mathbf i')+\ell(i'_{s-1}).\]

As in the case of the Grassmannian Schubert varieties, the set of all Schubert varieties 
(with reference to the standard complete flag $\mathbb C\subset\mathbb C^2\subset \cdots\subset 
\mathbb C^n$) in 
$\mathbb CG(\nu)$ and the set of $T$-fixed points of $\mathbb CG(\nu)$ 
are in bijective correspondence with $I(\nu).$   An element $\mathbf i\in I(\nu)$ corresponds to  the $T$-fixed point $(E_{\mathbf i_1},E_{\mathbf i_2}, \ldots, E_{\mathbf i_s})\in \mathbb CG(\nu)$
and the corresponding Schubert variety is denoted $X(\mathbf i)$. 
An element $V=(V_1,\ldots, V_s)\in \mathbb CG(\nu)$ belongs to $X(\mathbf i)$ if and only 
if $V_1+\cdots +V_p$ belongs to $X(\mathbf j_p)\subset \mathbb CG_{n,m_p}$ for all $p<s$ where $\mathbf j_p\in I(n;m_p)$
is the sequence $(\mathbf i_1,\ldots, \mathbf i_p)$ rearranged in increasing order.
The Schubert variety $X(\mathbf i)\subset \mathbb CG(\nu)$ has (complex) dimension the length of $\mathbf i$. 
That is, $\dim_\mathbb CX(\mathbf i)=\ell (\mathbf i)$.  
See \cite[\S3.3]{lakshmibai-raghavan}, \cite[\S1]{brion}.  The dimension formula may be established for $s\ge 3$ 
by induction using the Grassmann bundle $\mathbb CG_{n,k}\hookrightarrow \mathbb CG(\nu)\to \mathbb CG(\nu')$ (with $\nu'$ as above) where $(V_1,\ldots, V_s)\in \mathbb CG(\nu)$ projects to 
$(V_1,\ldots, V_{s-2}, V_{s-1}+V_s)\in \mathbb CG(\nu').$

The Schubert cells have natural orientations arising from the fact that 
the complex vector spaces have canonical orientations.  In fact the open cell $\mathring{X}(\mathbf i)$ is isomorphic to the affine space $\mathbb C^{\ell(\mathbf i)}.$
We note that the conjugation map $\sigma:\mathbb CG(\nu)\to \mathbb CG(\nu)$ that sends $\mathbf L=(L_1,\ldots, L_s)$ to $\bar{\mathbf L}=(\bar L_1,\ldots, \bar L_s)$ 
stabilizes each Schubert cell.  In fact, under the identification $\mathring X(\mathbf i)\cong \mathbb C^{\ell(\mathbf i)},$ 
the $\sigma$ corresponds to the complex conjugation on $\mathbb C^{\ell (\mathbf i)}.$  In particular, 
$\sigma $ is orientation preserving (resp. reversing) on $X(\mathbf i)$ when $\ell(\mathbf i)$ is even (resp. odd). 

Since the real dimension of each Schubert cell is even, it follows that the differential of the cellular chain 
complex of $\mathbb CG(\nu)$ vanishes identically. Consequently, $H_*(\mathbb CG(\nu);\mathbb Z)\cong C_*(\mathbb CG(\nu);\mathbb Z)$ and $H^*(\mathbb CG(\nu);\mathbb Z)\cong C^*(\mathbb CG(\nu);\mathbb Z).$  
Thus $\mathbb CG(\nu)$ satisfies the hypotheses of Lemma \ref{conjugation-cw}.
The rank of $H^{2q}(\mathbb CG(\nu);\mathbb Z)$ equals the cardinality of $I_q(\nu):=\{\mathbf i\in I(\nu)\mid \ell(i)=q\}$.  

The fixed set $\fix(\sigma)\subset \mathbb CG(\nu)$ is identified with the real flag  manifold $\mathbb RG(\nu)=O(n)/(O(n_1)\times \cdots \times O(n_s))
$ via the embedding defined by the complexification $V\mapsto V\otimes \mathbb C\subset \mathbb C^n$ where $V$ is a real vector space $V\subset \mathbb R^n$.  

\subsection{A cell structure on $P(m,\nu)$}
We put the standard $\mathbb Z_2$-equivariant cell structure $\{C_k^\pm\}_{0\le k\le m}$ on the sphere $\mathbb S^m$. Here the equivariance is with respect to the cyclic group generated by the antipodal map $\alpha:\mathbb S^m\to\mathbb S^m$.  The cell $C_k^+ $ (resp. $C_k^-$) is the upper (resp. lower)  hemisphere 
of $\mathbb S^k\subset \mathbb S^m$ where 
$\mathbb S^k\subset \mathbb S^m$ is the unit sphere in $\mathbb R^{k+1}$ spanned by $e_j, 1\le j\le k+1.$   
The projection $p:\mathbb S^m\to \mathbb RP^m$ is cellular with respect to the standard cell structure on $\mathbb RP^m$.   The unique closed $k$-cell in $\mathbb RP^m$ is 
$\mathbb RP^k$ corresponding the inclusion $\mathbb R^{k+1}\hookrightarrow \mathbb R^{m+1}.$

We put the standard orientation on $C_j^+$ for each $j$ and put the orientation 
on $C_j^-$ induced on it via
$\alpha$. The cells $C_j$ in $\mathbb RP^m$ are given the orientation induced from $C_j^+$ via $p|_{C_j^+}:C_j^+\to C_j$.  

A cell decomposition of $\mathbb S^m\times \mathbb CG(\nu)$ which is 
equivariant with respect to $\theta =\alpha\times \sigma$ is obtained 
by taking the product cells $X^\pm(j,\mathbf i):=C^\pm_j\times X(\mathbf i)$ 
as $(j,\mathbf i)$ varies in $I(m,\nu):=\{j\mid 0\le j\le m\}\times I(\nu)\}$.  
The cell $C^\pm_j\times X(\mathbf i)$ is given the product orientation.

We shall denote the image of the oriented cell $X^+(j,\mathbf i)$ under 
the double covering projection 
$\pi:\mathbb S^m\times \mathbb CG(\nu)\to P(m,\nu)$ by $X(j,\mathbf i)$ and put the induced orientation on it.  The deck transformation group of $\pi$ is generated by $\theta=\alpha\times \sigma.$

We note that $\theta|_{X^+(j,\mathbf i)}:X^+(j,\mathbf i)\to X^-(j,\mathbf i)$ is orientation preserving if and only if $\sigma|_{X^+(\mathbf i)}:X^+(\mathbf i)\to X^-(\mathbf i)$ is, if and 
only if $\ell(\mathbf i)$ is even. 

Since the boundary map of $C_*(\mathbb CG(\nu);\mathbb Z)$ vanishes,  it is readily seen that the boundary map of $C_*(\mathbb S^m\times \mathbb CG(\nu))$ is determined by $\partial:C_*(\mathbb S^m;\mathbb Z)$ and we obtain, for all $(j, \mathbf i)\in I(m, \nu),$ that 
\begin{equation}\label{boundary-product}
\partial (X^\pm(j,\mathbf i))= X^{\pm}(j-1,\mathbf i)+(-1)^{j}X^\mp(j-1,\mathbf i),
\end{equation}
and 
\begin{equation}\label{theta-star}
\theta_*(X^\pm(j,\mathbf i))=(-1)^{\ell(\mathbf i)}X^\mp(j, \mathbf i).
\end{equation}

Now $p_*(X^+(j,\mathbf i))=X(j,\mathbf i)$ and $p_*(X^-(j,\mathbf i))=p_*(\theta_*((-1)^{\ell(\mathbf i)}X^+(j,\mathbf i)))=p_*((-1)^{\ell(\mathbf i)}X^+(j,\mathbf i))=
(-1)^{\ell(\mathbf i)}X(j,\mathbf i)$.
 It follows from Equation \ref{boundary-product} that 
 \begin{equation}\label{boundary-pmx}
 \partial X(j,\mathbf i)=p_*\partial (X^+(j,\mathbf i))=
 (1+(-1)^{j+\ell(\mathbf i)})X(j-1,\mathbf i)~\forall (j,\mathbf i)\in I(m,\nu). 
 \end{equation}

 \subsection{(Co)homology of $P(m,\nu)$} 
 We now turn to the determination of the integral homology of $P(m,\nu)$.  
 It follows from Equation \ref{boundary-pmx} that, for $0<j\le m$, 
 $X(j,\mathbf i)$ is a cycle in $P(m,\nu)$ if and only if $j+\ell(\mathbf i)$ is odd. 
 When $j+\ell (\mathbf i)$ is odd and $j<m,$ we have  
 $\partial (X(j+1,\mathbf i))=2X(j,\mathbf i)$ and so the 
 homology class $[X(j,\mathbf i)]$ has order $2$.

It is evident from Equation \ref{boundary-pmx} that $X(0,\mathbf i)$ is a cycle for all $\mathbf i\in I(\nu).$ When 
$\ell(\mathbf i)$ is even, $[X(0,\mathbf i)]$ is of infinite order. 
If $\ell(\mathbf i)$ is odd, $X(0,\mathbf i)$ is of order 2.  Similarly, if $m$ is odd and $\ell(\mathbf i)$ is even, 
$[X(m,\mathbf i)]$ is of infinite order. If both $m $ and $\ell(\mathbf i)$ are even, then $X(m,\mathbf i)$ is not a cycle.  If $m$ is even and $\ell(\mathbf i)$ is odd, then $[X(m,\mathbf i)]$ is of infinite order.

Let $q\ge 0$.  Define the sets 
\[
\begin{array}{rll}
I_e(\nu) &:=&
\{\mathbf i\in I(\nu)\mid \ell(\mathbf i)\equiv 0\pmod 2\},\\ 
I_o(\nu)&:=& \{\mathbf i\in I(\nu)\mid \ell(\mathbf i)\equiv 1\pmod 2\}=I(\nu)\setminus I_e(\nu),\\
\mathcal B_{2q}&:=&\{[X(0,\mathbf i)]\mid q=\ell(\mathbf i),  \mathbf i\in I_e(\nu)
\}\cup \{[X(m,\mathbf i)]\mid 2q=m+2\ell(\mathbf i), \mathbf i\in I_o(\nu)\},\\
\mathcal B_{2q+1}&:=&\{[X(m,\mathbf i)]\mid 2q+1=m+2\ell(\mathbf i), \mathbf i\in I_e(\nu) \},\\
\mathcal B'_{q}&:=&\{ [X(j, \mathbf i)]\mid q=j+2\ell(\mathbf i), j+\ell(\mathbf i)\equiv 1\pmod 2,  0\le j<m,~ \mathbf i\in I(\nu)\}.\\ 
\end{array}
\]
We note that if $m$ is even, then $\mathcal B_{2q+1}=\emptyset$, and if $m$ is odd, then  
$\{[X(m,\mathbf i)]\mid 2q=m+2\ell(\mathbf i), \mathbf i\in I_e(\nu)\}=\emptyset $.

Observe that the cycle group $Z_q(P(m,\nu);\mathbb Z)$ is the free abelian group 
generated by $\{X(j,\mathbf i)\}$ where $[X(j,\mathbf i)]\in \mathcal B_q\cup \mathcal B'_q$ 
for any $q\ge 0$.   In particular $H_q(P(m,\nu);\mathbb Z)$ is generated by $\mathcal B_q\cup 
\mathcal B'_q$.

We have the following result where we leave out the known cases 
namely, $m=0$  as $P(0,\nu)\cong \mathbb CG(\nu)$ and $\nu=(1)$ in which case 
$P(m,\nu)=\mathbb S^m.$ 

\begin{theorem}    \label{Integral-cohomolgy-groups-of-P(m,n,k)} 
We keep the above notations. Suppose that $m\ge 1$ and $n=|\nu|\ge 2.$  
Then there is no odd torsion in $H_*(P(m,\nu);\mathbb Z)$ and every 
$2$-torsion element is of order $2$.
 The set $\mathcal B_q$ is a basis for $H_q(P(m,\nu);\mathbb Z)/\mathrm{torsion}$ and $\mathcal B'_q$ is a $\mathbb Z_2$-basis for the torsion subgroup of $H_q(P(m,\nu);\mathbb Z)$.  
\end{theorem}
\begin{proof} 
We first show that $\mathcal B_q$ is $\mathbb Z$-linearly independent. 
Suppose that $Z=\sum a_{j,\mathbf i} X(j,\mathbf i)$ with $ a_{j,\mathbf i}\in \mathbb Z,$ is a boundary where the sum is over $(j,\mathbf i)\in I(m,\nu)$ such that  
$[X(j,\mathbf i)]\in \mathcal B_q,$ i.e. $Z=\partial C$ for some $C\in C_{q+1}(P(m,\nu))$.
Suppose that some $a_{j,\mathbf i}\ne 0$.  
Then there has to be a cell $X(k,\mathbf l)$ such that 
$X(j,\mathbf i)$ occurs with non-zero coefficient in the expression of 
$\partial X(k,\mathbf l)$ for some $X(k,\mathbf l)$.  However, Equation \ref{boundary-pmx} shows 
that the only possibility is $k=j+1\le m, \mathbf l=\mathbf i$ and that 
$\partial X(j+1,\mathbf i)=(1+(-1)^{j+1+\ell(\mathbf i)})X(j,\mathbf i)$.  As $[X(j,\mathbf i)]\in \mathcal B_q$, we have $\partial X(j+1,\mathbf i)=0$.  This shows that there is no such $C$ and so 
the $\mathcal B_q$ is $\mathbb Z$-linearly independent. 

Note that 
$\mathcal B_q\subset \textrm{Fix}(\theta_*).$ In view of this
one can also prove the linear independence of $\mathcal B_q$ using 
Proposition \ref{homology-integral-fixed}.   

Since $\mathcal B_q\cup \mathcal B'_q$ generates $H_q(P(m,\nu);\mathbb Z)$ as observed 
above, it follows that $\mathcal B_q$ generates the abelian group  
$H_q(P(m,\nu);\mathbb Z)/\mathrm{torsion}.$  Therefore $\mathcal B_q$ 
is a $\mathbb Z$-basis for $H_q(P(m,\nu);\mathbb Z)$ mod torsion.   It is easily seen that $\mathcal B_q'$ is $\mathbb Z_2$-linearly independent.

Any torsion element of $H_q(P(m,\nu);\mathbb Z)$ is of order $2$, in view of the fact that 
$H_q(P(m,\nu);\mathbb Z)$ is generated by $\mathcal B_q\cup \mathcal B'_q$.  Indeed, $\mathcal B_q$
generates a free abelian subgroup of $H_q(P(m,\nu);\mathbb Z)$ and $\mathcal B'_q$ generates 
an elementary abelian $2$-group contained in $H_q(P(m,\nu);\mathbb Z)$. Hence 
the torsion subgroup equals the subgroup generated by $\mathcal B'_q$ and our assertion follows.
\end{proof} 

  The cohomology groups of $P(m,\nu)$ can be read off from the above theorem using the universal coefficient theorem.  Thus $H^q(P(m,\nu);\mathbb Z)\cong \mathbb Z^t\oplus\mathbb Z_2^{p}$ where the values of $t$ and $p$ depend on $m,\nu$ can be explicitly determined in terms of $|\mathcal B_q| $ and 
  $|\mathcal B'_{q-1}|.$  As a warm up, we compute $H^2(P(m,\nu);\mathbb Z)$.  
  Note that 
  $\mathcal B_2=\emptyset$ and $\mathcal B_1'=\{[X(1,\mathbf i_0)]\}$ where $\mathbf i_0
  =(1,2\,\ldots, n)$ 
  is the identity permutation (which is the unique element of length $0$) if $m>1$ and $\mathcal B_1'$ is empty if $m=1$.
So the rank of $H_2(P(m,\nu);\mathbb Z)$ is zero.  Also 
  $H_1(P(m,\nu);\mathbb Z)\cong \mathbb Z_2$ if $m>1$ and is isomorphic to $\mathbb Z$
  if $m=1.$ 
  \begin{proposition}
      Let $m\ge 1.$  With the above notations,  
      \[ H^2(P(m,\nu);\mathbb Z)\cong \begin{cases}
          \mathbb Z_2&~\textrm{if~} m>1\\
          0&~\textrm{if~} m=1.\\
                \end{cases} \]
Also the Picard group $Pic(P(m,\nu))\cong \mathbb Z_2$, generated by $\xi_\mathbb C$ if $m>1$ and is 
trivial otherwise.  
  \end{proposition}
\begin{proof}
    The first assertion follows from the universal coefficient theorem. The second assertion 
    follows from the fact that the first Chern class map yields an isomorphism of the 
    Picard group onto $H^2(P(m,\nu);\mathbb Z).$
\end{proof}

 We now turn to the determination of the rank and the torsion subgroup of 
  the integral cohomology groups of $P(m,\nu)$ explicitly in terms of $m,\nu$. 

  Write  
  $H^{\textrm{ev}}(P(m,\nu);\mathbb Z)
  =\bigoplus_{q\ge 0}H^{2q}(P(m,\nu);\mathbb Z)\cong \mathbb Z^{b_e}\oplus \mathbb Z^{b'_e}_2$
  and $H^{\mathrm{odd}}(P(m,\nu);\mathbb Z)=\bigoplus_{q\ge 1}H^{2q-1}(P(m,\nu);\mathbb Z)\cong \mathbb Z^{b_o}\oplus \mathbb Z_2^{b'_o}$.  
  We shall obtain formulas for the numbers $b_e,b_e',b_o,b'_o$ in terms of $m,\nu.$

Set $\mathcal B:=\bigcup_{q\ge 0}\mathcal B_q, \mathcal B_e:=\bigcup_{q\ge 0} \mathcal B_{2q}$ and 
let $\mathcal B_o=\mathcal B\setminus \mathcal B_e$.  Likewise $\mathcal B':=\bigcup_{q\ge 1}\mathcal B_q'$ and  
$\mathcal B'_e:=\{[X(j,\mathbf i)]\in \mathcal B'\mid j\equiv 0\pmod 2\},$ and $\mathcal B'_o=\mathcal B'\setminus \mathcal B'_e.$   
Also, set 
$\beta_e:=|\mathcal B'_e|, \beta_o:=|\mathcal B_o'|.$

By the universal coefficient theorem, $b_e=|\mathcal B_e|, b_o=|\mathcal B_o|$, whereas 
$b_e'=\beta_o,  b'_o=\beta_e.$

Set $\ell_e:= |I_e(\nu)|$ 
and $\ell_o=|I_o(\nu)|$.
Thus $\ell_e+\ell_o={n\choose \nu}=\chi(\mathbb CG(\nu)).$
We shall determine the values of $\ell_e, \ell_o$.

The Poincaré polynomial $P_t(\nu)$ of $\mathbb CG(\nu)$ is the $q$-Gaussian multinomial coefficient $Q_q(\nu):={n\choose \nu}_q=[n]_q/([n_1]_q\cdots[n_s]_q)$ where  
$q:=t^2$ and $[k]_q=\prod_{1\le j\le k}
(1-q^j).$  
The right hand side of $Q_q$ is also the Poincaré polynomial $P_q(\mathbb RG(\nu);\mathbb Z_2)$ of the 
{\em real} flag manifold $\mathbb RG(\nu)$ with 
$\mathbb Z_2$-coefficients.  See \cite[\S 26]{borel},  \cite{borel-cmh}.
Therefore
\begin{equation}\label{le-lo}
\ell_e-\ell_o
=\chi(\mathbb RG(\nu)).
\end{equation}
We conclude that 
\begin{equation}
\ell_e=(\chi(\mathbb CG(\nu))+\chi(\mathbb RG(\nu)))/2; ~~\ell_o=(\chi(\mathbb CG(\nu))-\chi(\mathbb RG(\nu)))/2.
\end{equation}

If $\chi(\mathbb RG(\nu))=0,$ then $\ell_e=\ell_o={n\choose \nu}/2$. 

The Euler-Poincaré characteristic of $\mathbb RG(\nu)$ is can be computed using 
standard arguments and the classical result that $\chi(G/T)=|W(G,T)|$, the order of 
the Weyl group $W(G,T)$ when $G$ is a compact connected Lie group and $T$ is 
a maximal torus. We include a proof for the sake of completeness as we could not find an explicit reference for the result.

\begin{lemma} \label{chi-real-flag-manifolds}
    Let $\nu=(n_1,\ldots,n_s), s\ge 2.$  Set $\nu_o:=\{1\le j\le s\mid n_j\equiv 1\mod 2\}.$
    
    Then $\chi(\mathbb RG(\nu))=0$ if $\nu_o\ge 2.$
    Suppose that
    $\nu_o\le 1.$  Set $n'_j:=\lfloor n_j/2\rfloor, 1\le j\le s,$ and 
    $\lfloor\nu/2\rfloor:=
    (n_1',n_2',\ldots, n_s')$.  Then $\chi(\mathbb RG(\nu))={\lfloor n/2\rfloor\choose 
    \lfloor\nu/2\rfloor}$.
\end{lemma}
\begin{proof}  Without loss of generality, assume that $n_j$ is even for all $j>\nu_o$.  
     When $\nu_o\ge 2$, 
    $\mathbb RG(\nu)$ is fibred by the real Grassmann manifold $\mathbb RG(n_1,n_2)$ over $\mathbb RG(n_1+n_2,n_3,\ldots, n_s).$ Since $\dim \mathbb RG(n_1,n_2)=n_1n_2$ is odd, $\chi(\mathbb RG(\nu))=0$ and so the vanishing of $\chi(\mathbb RG(\nu))$
    follows from the multiplicative property of the Euler-Poincaré characteristics for fibrations.  See \cite[Theorem 1,\S3, Chapter 9]{spanier} 

    Let $\nu_o\le 1$.  Consider first the Euler-Poincaré characteristic of the 
    {\em oriented} flag manifold $\widetilde{\mathbb RG}(\nu):=SO(n)/SO(\nu)$ where 
    $SO(\nu):=SO(n_1)\times\cdots \times SO(n_s).$  As $\nu_o\le 1$, $SO(\nu)$ has the same rank as $SO(n).$  
    We take $T=(SO(2))^{\lfloor n/2\rfloor}\subset SO(\nu)$ to be the standard maximal torus of $SO(n)$ which sits in $SO(n)$ as $2\times 2$ block diagonal with top diagonal entry being $1$ if $n$ is odd.   
    Using the fibration $SO(n)/T\to SO(n)/SO(\nu)$ with fibre $SO(\nu)/T$, we obtain 
    that $\chi(SO(n)/T)=\chi(SO(\nu)/T)).\chi(SO(n)/SO(\nu))$. 
    It is a classical result that when $G$ is a compact connected Lie group and $T\subset G$ is a maximal torus of $G$, then 
    $\chi(G/T)=|W(G,T)|$, the order of the Weyl group of $G$ with respect to $T$. 
    See \cite[Part-II]{mimura-toda}.
    It is known \cite{husemoller} that $W(SO(n),T)$ has order $2^{\lfloor (n-1)/2\rfloor}\cdot \lfloor n/2\rfloor !.$ This yields that $\chi(\widetilde{\mathbb RG}(\nu))=2^{s-1} {\lfloor n/2\rfloor
    \choose \lfloor \nu/2\rfloor}$.  The value of $\chi(\mathbb RG(\nu))$ is 
    then determined to be ${\lfloor n/2\rfloor
    \choose \lfloor \nu/2\rfloor}$ using the covering projection 
    $\widetilde{\mathbb RG}(\nu)\to \mathbb RG(\nu)$, which has degree $2^{s-1}.$
\end{proof}

With notations as in the above lemma, if $\nu_o\le 1$, then  
$\ell_e=({n\choose \nu}+{\lfloor n/2\rfloor \choose \lfloor \nu/2\rfloor })/2$ and $\ell_o=({n\choose \nu}-{\lfloor n/2\rfloor\choose \lfloor \nu/2\rfloor })/2.$

We tabulate below the values of $\ell_e$ and $\ell_o$. 
\[  
\begin{array}{|c||c|c|} 
 \hline
 \nu & \ell_e  &\ell_o  \\
 \hline\hline
 \nu_o\geq 2& {n\choose \nu}/2  &  {n\choose \nu}/2   \\
 \hline
 \nu_o\leq 1 & ({n\choose \nu} +{\lfloor n/2\rfloor\choose \lfloor\nu/2\rfloor})/2    &
 ({n\choose \nu} -{\lfloor n/2\rfloor\choose \lfloor\nu/2\rfloor})/2   \\
 \hline
\end{array}
\]
\begin{center}
Table 1: The values of $\ell_e$ and $\ell_o$.
\end{center}

From the definition of $\mathcal B'_e$ we see that if $\mathbf i\in I_o(\nu)$, then $[X(j,\mathbf i)]\in \mathcal B'_e$ for $\lfloor (m+1)/2\rfloor$ distinct values of $j$, namely, those which satisfy $j\equiv 0\pmod 2, 0\le j<m.$  Moreover, if $[X(j,\mathbf i)]\in \mathcal B'_e,$ then $\mathbf i\in I_o(\nu).$
Therefore $\beta_e=\lfloor (m+1)/2\rfloor\cdot  
\ell_o$.  Similarly $\beta_o=\lfloor m/2\rfloor\cdot  
\ell_e $.

The values of $b_e,b_o$ depend on $\nu$ and the parity of $m$. We proceed to determine them. 

We have $b_e+b_o=\dim_\mathbb Q(H^*(P(m,\nu);\mathbb Q))$ and $b_e-b_o=\chi(P(m,\nu)).$

{\em Case 1:} \underline{$m$ is even.}   Then $H^q(P(m,\nu);\mathbb Q)=0$ if $q$ is odd and so $b_o=0$. It follows that $b_e=\chi(P(m,\nu))=(1/2)\chi(\mathbb S^m\times \mathbb CG(\nu))={n\choose \nu}=n!/(n_1!\cdots n_s!).$

{\em Case 2:} \underline{$m$ is odd.} We have 
and $b_e-b_o=\chi(P(m,\nu))=0.$  Therefore $b_e=b_o.$  Also $b_e=|\{\mathbf i\in I(\nu)\mid 
\ell(\mathbf i)\equiv 0\mod 2\}|=\ell_e$.

We summarise the above results in the theorem below. 

\begin{theorem} \label{betti-torsion-coeffecients}
   Let $m\ge 1$ and $\nu=(n_1,\ldots, n_s)$ where $s\ge 2$.  Let $\nu_o:=|\{1\le j\le s\mid n_j\equiv 1\pmod 2\}|.$
   Then  $H^{\mathrm{ev}}(P(m,\nu);\mathbb Z)\cong 
   \mathbb Z^{b_e}\oplus \mathbb Z_2^{b'_e}$ and 
   $H^{\mathrm{odd}}(P(m,\nu);\mathbb Z)\cong 
   \mathbb Z^{b_o}\oplus \mathbb Z_2^{b'_o}$ where $b_e,b_o,b'_e,b'_o$ are as follows:\\
(i) When $m$ is odd,  
\[b_e=b_o=\ell_e=\begin{cases}
    {n\choose \nu}/2, & ~\mathrm{if~} \nu_o\ge 2 \\
    ({n\choose \nu}+{\lfloor n/2\rfloor \choose \lfloor \nu/2\rfloor})/2,& ~
    \mathrm{if~} \nu_o\le 1.\\
\end{cases}\]
When $m$ is even, $b_e={n\choose \nu}$ and $b_o=0$.\\
(ii)  $ b'_e=\beta_o=\lfloor m/2\rfloor\cdot  
\ell_e,~~~b'_o=\beta_e=\lfloor (m+1)/2\rfloor\cdot  
\ell_{o}.$

  In particular, $H^{ev}(P(1,\nu);\mathbb Z)$ has no 
    torsion. \hfill $\Box$
\end{theorem}

The following remark gives some partial information about the 
ring structure of the integral cohomology of $P(m,\nu)$.

\begin{remark}(i)  It is readily seen that $\mathbb CG(\nu)$ equals the Schubert variety $X(\mathbf w_0)$ that corresponds to the permutation 
$w_0=(n,n-1,\ldots,2, 1)\in S_n$ (in the one-line notation).  So $\sigma$ is orientation preserving if and only if 
$d=\dim_\mathbb C\mathbb CG(\nu)=\sum_{1\le i<j\le s}n_in_j$ is even, if and only if ${ \nu_o\choose 2}$ is even.  It follows that 
$P(m,\nu)$ is orientable if and only if $m+{\nu_o \choose 2}$ is odd. (See Remark \ref{orientability}.) When $P(m,\nu)$ is orientable,  the natural orientation 
on it corresponds to the fundamental class $[X(m,\mathbf w_0)].$

(ii)
 Since the bundle $p:P(m,\nu)\to \mathbb RP^m$ admits a cross-section $s:\mathbb RP^m\to P(m,\nu)$,
$p^*:H^*(\mathbb RP^m;\mathbb Z)\to H^*(P(m,\nu);\mathbb Z)$ is a monomorphism of rings 
and its image is a direct summand. 

(iii) Let $\iota:\mathbb CG(\nu)\hookrightarrow P(m,\nu)$ denote the  fibre-inclusion of the bundle $p: P(m,\nu)\to \mathbb RP^m$ over a point, say $[e_1]\in \mathbb RP^m.$ 
Since $\{ [X^+(0,\mathbf i)]\mid \mathbf i\in I_e(\nu)\}$ is a $\mathbb Z$-basis for $\bigoplus_{q\ge 0}H_{4q}(\mathbb CG(\nu);\mathbb Z)\subset H_*(\mathbb S^m\times \mathbb CG(\nu);\mathbb Z)$ and since 
$\iota_*([X(\mathbf i)])= \pi_*[X^+(0,\mathbf i)]=[X(0,\mathbf i)]$, it follows that $\iota_*(\bigoplus H_{4q}(\mathbb CG(\nu);\mathbb Z))\cong\mathbb Z^{\ell_e} $ is a direct summand of $H_{4q}(P(m,\nu);\mathbb Z).$
Therefore 
\begin{equation}\label{iota^*}
\iota^*: \bigoplus_{q\ge 0}H^{4q}(P(m,\nu);\mathbb Z)\to \bigoplus_{q\ge 0}H^{4q}(\mathbb CG(\nu);\mathbb Z)
\end{equation}
is a surjective ring homomorphism. Since $\bigoplus_{q\ge 0}H^{4q}(\mathbb CG(\nu);\mathbb Z)$ is torsion free, the kernal of $\iota^*$ contains the torsion subgroup of $\bigoplus_{q\ge 0}H^{4q}(P(m,\nu);\mathbb Z)$. \\

(iv) 
Let $\{x(m,\mathbf i)\} $ in $Hom(H_*(P(m,\nu);\mathbb Z),\mathbb Z) $ be dual to $[X(m,\mathbf i)]$
with respect to the basis $\mathcal B=\bigcup_{q\ge 0} \mathcal B_q$ of $H_*(P(m,\nu);\mathbb Z)/\mathrm{torsion}$.
We have \[\begin{array}{rcl}
\langle \iota^*(x(m,\mathbf i)),[X(\mathbf j)]\rangle &=&\langle x(m,\mathbf i), \iota_*([X(\mathbf j)])\rangle \\&=& \langle x(m,\mathbf i),[X(0,\mathbf j)]\rangle\\&=&0,\\
\end{array}
\]
for all $[X(\mathbf j)]\in H_*(\mathbb CG(\nu);\mathbb Z)$.  
Thus the kernel $\iota^*:H^*(P(m,\nu);\mathbb Z)\to H^*(\mathbb CG(\nu);\mathbb Z)$ contains $x(m,\mathbf i)$.

(v)   
The cup-product $x(m,\mathbf i)\smile x(m,\mathbf j)$ is a torsion element.  This is 
obvious when any one of the factors is a torsion element. 
Suppose that both are of infinite order.  Now $\pi^*(x(m,\mathbf i)), 
\pi^*(x(m, \mathbf j))\in H^m(\mathbb S^m;\mathbb Z)\otimes H^*(\mathbb CG(\nu);\mathbb Z).$ Hence $x(m,\mathbf i)\smile x(m,\mathbf j)\in \ker \pi^*$. 
By Proposition \ref{nooddtorsion}, $\ker(\pi^*)$ consists only of torsion elements, 
and so our assertion follows.

\end{remark}

\subsection{The ring structure of $H^*(P(m,\nu);R)$}
Let $R$ be a commutative ring in which $2$ is a unit.
From Theorem \ref{Integral-cohomolgy-groups-of-P(m,n,k)} we know that 
as an $R$-module, $H^q(P(m,\nu);R)$ is free 
of rank $|\mathcal B_q|.$
We now turn to the ring structure of $H^*(P(m,\nu);R)$.  We shall first describe 
it as a subring of $H^*(\mathbb S^m\times \mathbb CG(\nu);R)$ via the 
homomorphism $\pi^*$ induced by the double covering projection $\pi:\mathbb S^m\times 
\mathbb CG(\nu)\to P(m, \nu).$  Note that since $2$ is a unit in $R$, $\pi^*$ is a monomorphism; see Proposition \ref{nooddtorsion}. Then, in Theorem \ref{cohomologyringofPmnu} we shall 
describe $H^*(P(m,\nu);R)$ as a quotient of a polynomial algebra. 

It will be convenient 
to use Chern classes of the canonical bundles over $\mathbb CG(\nu)$ as ring generators of $H^*(\mathbb CG(\nu);\mathbb Z).$

Let $1\le j\le s$ and let $\gamma_j$ denote the complex $n_j$-plane bundle over $\mathbb CG(\nu)$ whose fibre over $\mathbf L=(L_1,\ldots, L_s)$ is $L_j.$ Then 
\[\bigoplus_{1\le j\le s} \gamma_j=n\epsilon_\mathbb C.\]
Denote by $c(\omega,t)$ the total Chern polynomial  $1+c_1(\omega)t+c_2(\omega)t^2+\cdots 
+c_r(\omega)t^r$ in the indeterminate $t$, where $c_j(\omega)$ is the $j$th Chern class of $\omega$ and $r$ is the rank of $\omega.$  The following relation holds in $H^*(\mathbb CG(\nu);\mathbb Z)$, in view of the above vector bundle isomorphism, where 
\begin{equation}\label{chernpolynomial-identity}
    \sum_{0\le r\le n}f_pt^p:= \prod_{1\le j\le s} c(\gamma_j,t)=1. 
\end{equation}
Recall that 
\[H^*(\mathbb C G(\nu);\mathbb Z)=\mathbb Z[c_{r,j}; 1\le r\le n_j, 1\le j\le s]/\langle f_1,\cdots, f_n\rangle \]
where the $f_j$ are defined by the equation 
\begin{equation}\label{chern-formal}
\sum_{0\le r\le n} f_rt^r=\prod_{1\le j\le s}
(1+c_{1,j}t+\cdots+c_{n_j,j}t^{n_j})
\end{equation}
under an isomorphism that maps $c_{r,j}$ to $c_r(\gamma_j)$.  See \cite[Proposition 31.1]{borel}.  

Denote by $u_m\in H^m(\mathbb S^m;\mathbb Z)$ the positive generator (with respect to the standard orientation on the sphere). 

We now describe, using Proposition \ref{Integral-cohomolgy-groups-of-P(m,n,k)}, the  
subring fixed by $\theta^*$ in $H^*(\mathbb S^m\times \mathbb CG(\nu);R)\cong 
R[u_m]\otimes R[c_{r,j}; 1\le r\le n_j, 1\le j\le s]/\langle u_m^2,\; f_p, 1\le p\le n\rangle$. Since $\pi^*$ defines an isomorphism of $H^*(P(m,\nu);R)\cong \fix(\theta^*)$, we have the following result. We omit the proof, which involves a straightforward verification.

\begin{theorem} \label{cohomologyofpmnu-as-subalgebra}
Suppose that $2$ is a unit in $R$. Then $H^*(P(m,\nu))$ is isomorphic 
to the subalgebra $\fix(\theta^*)\subset H^*(\mathbb S^m\times \mathbb CG(\nu);R)$
which is generated by the following elements: \\
{\it Case} (1): \underline{$m$ is even.} 
\[
\begin{array}{rll}
u_m c_{2p-1}(\gamma_r),& 1\le 2p-1\le n_r,~ 1\le r\le s ;\\
c_{2j}(\gamma_r),& 1\le j\le n_r,~ 1\le r\le s; \;\textrm{and}\\
c_{2p-1}(\gamma_i)c_{2q-1}(\gamma_j),& 1\le 2p-1\le n_i,~ 1\le 2q-1\le n_j,~ 1\le i\le j\le s.\\
\end{array}
\]
{\it Case} (2):  \underline{$m$ is odd.} 
\[
\begin{array}{rll}
u_m;\;\; c_{2j}(\gamma_r),& 1\le j\le n_r,~ 1\le r\le s;\;  \textrm{and}\\  
c_{2p-1}(\gamma_i)c_{2q-1}(\gamma_j),& 
1\le 2p-1\le n_i,~ 1\le 2q-1\le n_j,~ 1\le i \le j\le s.
\end{array}
\]
~\hfill $\Box$
\end{theorem}

Next, we shall present $H^*(P(m,\nu);R)$ as a quotient of a polynomial algebra over $R$.
We introduce commuting `variables' $u_m, c_{2p,i}, c'_{2p-1,i,}, 
c^{\prime\prime}_{2p-1,2q-1, i,j}$, 
$1\le 2p-1\le n_i, 1\le 2q-1\le n_j, 1\le i\le j\le s$. Their degrees are assigned 
as follows: $|u_m|=m, |c_{2p,i}|=4p, |c'_{2p-1,i}|=m+4p-2, |c^{\prime\prime}_{2p-1,2q-1,i,j}|=4(p+q-1)$.  

{\em It will be tacitly assumed that $c^{\prime\prime}_{2p-1,2q-1,i,j}
$ is the same as the variable $ c^{\prime\prime}_{2q-1,2p-1,j,i}. $} 

Recall the polynomials $f_r, 1\le r\le n,$ defined in Equation \ref{chern-formal}. When $r=2r'$ is even, one may express $f_r$ as a polynomial in $c_{2p,i}, c_{2p-1,2q-1,i,j}$, possibly in more than one way.  Choose one such expression and define $F_r$ to be the resulting polynomial.

When $m$ is even, the same procedure applied to $c_{2t-1,l}f_r, 2p-1\le n_l,$ 
 for $r$ odd yields the polynomials $F_{r, 2t-1,l}$ in the variables  
$ c_{2q,j}, c'_{2p-1,i}, c^{\prime\prime }_{2p'-1,2q'-1,i',j'}$, 
$1\le 2p'-1,2p-1 \le n_{i'},
 1\le 2q'-1,2q\le n_j,   1\le i,j \le s.$ 
   Similarly, assuming that $r$ is odd, 
 $u_m f_r$ can be expressed as a polynomial in the variables $c_{2q,j}, c'_{2p-1,i}, c^{\prime\prime}_{2p'-1,2q'-1,i',j'}$ in several ways.  
 We choose one such expression and denote the resulting polynomial by  
 $F'_{r,m}.$

\begin{theorem} \label{cohomologyringofPmnu} We keep the above notations. 
Let $R$ be a commutative ring in which $2$ is a unit.  Then $H^*(P(m,\nu);R)$ is isomorphic, as a graded $R$-algebra to $\mathcal R/\mathcal J$ where $\mathcal R$ is a polynomial algebra $\mathcal R$ whose description 
depends on the parity of $m$ and is defined below.

{\em Case (1)}. Suppose that $m\equiv 0\pmod 2$. Then 
\begin{center}
    $\mathcal R: =R[c_{2t,i};\; c'_{2p-1,i};\; c^{\prime\prime}_{2p-1,2q-1, i,j}; ~1\le 2t, 
2p-1\le n_i,~ 1\le 2q-1\le n_j,~ 1\le i,j\le s], $
\end{center} 
with $|c_{2t,i}|=4t, |c_{2p-1,i}'|=m+4p-2, |c^{\prime\prime}_{2p-1,2q-1,i,j}|=4p+4q-4.$
The ideal $\mathcal J\subset \mathcal R$ is the ideal generated by the following elements:\\
(i) $c'_{2p-1,i}c'_{2q-1,j}; ~~ c'_{2p-1,i}
c^{\prime\prime}_{2q-1,2q'-1,j,j'}-c'_{2q'-1,j'}c^{\prime\prime}_{2p-1,2q-1,i,j}$;\\ 
(ii) $c^{\prime\prime}_{2p-1,2q-1, i,j}c^{\prime\prime}_{2p'-1,2q'-1,i',j'}-c^{\prime\prime}_{2p-1,2p'-1,i,i'}c^{\prime\prime}_{2q-1,2q'-1, j,j'}$;\\
(iii) $F_{2r},~ F'_{2k-1,m}, ~F_{2k-1, 2t-1,l},~  1\le 2k-1, 2r \le n,~ 1\le 2t-1\le n_l,~ 1\le l\le s $;\\
where  $2p-1\le n_i,~ 2p'-1\le n_{i'},~ 2q-1\le n_j,~ 2q'-1\le n_{j'},~ 1\le i,j,i',j'\le s$.\\
The isomorphism $\eta:\mathcal R/\mathcal J\to H^*(P(m,\nu);R)$ is established by 
\[
c_{2p,i}\mapsto c_{2p}(\gamma_i),\;\;
c'_{2p-1,i}\mapsto u_mc_{2p-1}(\gamma_i),\;\;
 c^{\prime\prime}_{2p-1,2q-1,i,j}\mapsto c_{2p-1}(\gamma_i)c_{2q-1}(\gamma_j).
\]

{\em Case (2)}. Suppose that $m\equiv 1\pmod 2$.  Then 
\begin{center}
    $\mathcal R:=R[u_m;\; c_{2p,i},1\le 2p\le n_i;\;\;  c^{\prime\prime}_{2p-1,2q-1,i,j},~ 2p-1\le n_i,~ 2q-1\le n_j,~ 1\le i,j\le s],$
\end{center}
with $|u_m|=m, |c_{2p,i}|=4p, |c_{2p-1,2q-1,i.j}^{\prime\prime}|=4p+4q-4$.  
 The ideal $\mathcal J\subset \mathcal R$ is generated by 
the following elements:\\
(i) $u_m^2$;~~ $c^{\prime\prime}_{2p-1,2q-1,i,j}c^{\prime\prime}_{2p'-1,2q'-1,i',j'}-
c^{\prime\prime}_{2p-1,2p'-1,i,i'}c^{\prime\prime}_{2q-1,2q'-1,j,j'}$;\\ 
(ii) $F_{2r}, ~F_{2k-1, 2t-1, l},~ 1\le 2t-1\le n_l,~ 1\le l\le s;$\\
where  $2p-1\le n_i,~ 2p'-1\le n_{i'},~ 2q-1\le n_j,~ 2q'-1\le n_{j'},~ 1\le i,j,i',j'\le s$.\\
An isomorphism $\eta: \mathcal R/\mathcal J\to H^*(P(m,\nu);R)$ is obtained as 
\[
u_m\mapsto u_m,\;\; c_{2p,i}\mapsto c_{2p}(\gamma_i),\;\; 
c^{\prime\prime}_{2p-1,2q-1,i,j}\mapsto c_{2p-1}(\gamma_i)c_{2q-1}(\gamma_j).
\] 
\end{theorem}
\begin{proof}
   
It is clear $\eta$ is a well-defined $R$-algebra homomorphism. 
Also image of $\eta$ equals $\fix(\theta^*).$
Since $2$ is a unit in $R$, we have $H^*(P(m,\nu);R)\cong \fix(\theta^*)\subset H^*(\mathbb S^m\times \mathbb CG(\nu);R)$ and so, to complete the proof, 
we need only show that it is a monomorphism of $R$-modules.

Let $\tilde{\mathcal R}_0$ be the polynomial algebra $\widetilde{\mathcal R}_0=R[c_{q,i}\mid 1\le q\le n_i; 1\le i\le s]$ and let $\mathcal R_0$ be the $R$-subalgebra of $\tilde{\mathcal R}_0$ 
generated by the $c_{2q,i},\;\; 
c_{2p-1,2q-1,i,j}:=
c_{2p-1,i}c_{2q-1,j}$.  The algebra $\tilde{\mathcal R}_0$
is graded by $|c_{q,i}|=2q~\forall i.$   
Let $\tilde\sigma:\tilde{\mathcal R}_0\to \tilde{\mathcal R}_0$ be the $R$-algebra involution defined by $c_{q,i}\mapsto (-1)^qc_{q,i}$.
Then $\mathcal R_0=\bigoplus_{t\ge 0} \widetilde{R}^{4t}
=\fix(\tilde\sigma)$. 
The cohomology algebra 
$H^*(\mathbb CG(\nu);R)$ is isomorphic to 
$\widetilde{R}_0/\tilde{\mathcal J}_0$ where $\tilde{\mathcal J}_0$
is generated by $f_1,\cdots,f_n$ defined as in Equation \ref{chernpolynomial-identity}.
The isomorphism is obtained by sending $c_{q,i}$ to $c_q(\gamma_i)\in H^{2q}(\mathbb CG(\nu);R).$

The ideal $\tilde{\mathcal J}_0$ is graded and we have that $\tilde{ \mathcal J}_0\cap \mathcal R_0
=\oplus_{t\ge 0}\tilde{\mathcal J}_0^{4t}$.  
So $\mathcal J_0:=\mathcal R_0\cap \tilde{ \mathcal J}_0\subset \mathcal R_0$
is generated as an ideal by the following elements:
\[f_{2r},\; c_{2p-1,i}f_{2r-1}\in \fix(\tilde\sigma), 1\le 2p-1\le n_i,1\le i\le s.\]  We may express each of these elements as polynomials in the chosen generators of $\mathcal R_0$.  The non-uniqueness of such expressions for a given ideal generator arises from the equality 
$c_{2p-1,2q-1,i,j}c_{2p'-1,2q'-1,i',j'}=
c_{2p-1,2p'-1,i,i'}c_{2q-1,2q'-1,j,j'}.$

Consider the polynomial algebra over $R$ generated by the indeterminates $c_{2p,i}, c_{2p-1,2q-1,i,j}^{\prime\prime}$.   Then the quotient 
algebra 
\begin{center}
$R[c_{2p,i}, 
c_{2p-1,2q-1,i,j}^{\prime\prime}]/
\langle c_{2p-1,2q-1,i,j}^{\prime\prime}c_{2p'-1,2q'-1,i',j'}^{\prime\prime}
-c_{2p-1,2p'-1,i,i'}^{\prime\prime}c_{2q-1,2q'-1,j,j'}^{\prime\prime}\rangle $ 
\end{center}
is isomorphic to $\mathcal R_0\subset \tilde{\mathcal R}_0$ via the $R$-algebra homomorphism 
\begin{center}
    $c_{2p,i}\mapsto c_{2p,i}, ~c_{2p-1,2q-1,i,j}^{\prime\prime}\mapsto c_{2p-1,2q-1,i,j}.$
\end{center}

Set $R_0:=R[c_{2p,i}, c_{2p-1,2q-1,i,j}^{\prime\prime}]$ and let $J_0\subset R_0$ be the ideal generated by 
\begin{center}
    $c_{2p-1,2q-1,i,j}^{\prime\prime}c_{2p'-1,2q'-1,i',j'}^{\prime\prime}
-c_{2p-1,2p'-1,i,i'}^{\prime\prime}c_{2q-1,2q'-1,j,j'}^{\prime\prime},\;\; 
F_{2r}, \;\;F_{2r-1, 2t-1,l}$.
\end{center}  Then 
\begin{equation}
    R_0/J_0
\cong \mathcal R_0/\mathcal J_0\cong H^{4*}(\mathbb CG(\nu);R)\cong \fix(\sigma^*).
\end{equation}

The ring $\mathcal R/\mathcal J$ is an $\mathcal R_0/\mathcal J_0$-module whose 
(definition and) structure depends on the parity of $m$.  

First, consider the case when $m$ is odd.  In this case, $\mathcal R/\mathcal J$ is a free $\mathcal R_0/\mathcal J_0$-module with basis $\{1, u_m\}$. 
Also, 
$H^*(\mathbb S^m;R)=R[u_m]/\langle u_m^2\rangle$ is fixed by $\alpha^*$.  

Since $u_m^2\in\mathcal{J},$ we have the following isomorphisms of $R$-algebras $\mathcal R/\mathcal J
\cong (\mathcal R_0/\mathcal J_0)[u_m]/\langle u^2_m\rangle \cong R[u_m]/\langle u_m^2\rangle \otimes_R(\mathcal R_0/\mathcal J_0)\cong H^*(\mathbb S^m;R)\otimes 
\fix(\sigma^*)\cong \fix(\theta^*)\cong H^*(P(m,\nu);R)$.  
This completes the proof 
when $m$ is odd.

Now let $m$ be even.  
We have $\alpha^*(u_m)=-u_m$ and  
\begin{equation} \label{m-even-fixtheta-star}
\fix(\theta^*)=\bigoplus_{t\ge 0}H^{4t}(\mathbb CG(\nu);R)\oplus \left( \bigoplus_{t\ge 0}Ru_m\otimes H^{4t+2}(\mathbb CG(\nu);R)\right).
\end{equation}
Since each $H^{4t+2}(\mathbb CG(\nu);R)$ is a free $R$-module for any $t\ge 0$,  so 
is $Ru_m\otimes H^{4t+2}(\mathbb CG(\nu);R)\cong H^{4t+2}(\mathbb CG(\nu);R).$

On the other hand, 
$\mathcal R/\mathcal J$ is generated as a
$\mathcal R_0/\mathcal J_0$-algebra by the 
$c'_{2p-1,i}$.  Since $c'_{2p-1,i}c'_{2q-1,j}=0$, we see that $\mathcal R/\mathcal J$
is generated as an $\mathcal R_0/\mathcal J_0$-module 
by $\{1, c'_{2p-1,i}\}.$ 
Further, in view of the relation $c'_{2p-1,i}c^{\prime\prime}_{2q-1,2q'-1,j,j'}=c'_{2q'-1,j'}c^{\prime\prime}_{2p-1,2q-1,i,j}$
in $\mathcal R/\mathcal J,$ if $x\in \mathcal R/\mathcal J$ is a monomial in the generators $c_{2p-1,2q-1,i,j}^{\prime\prime}$ and if $r$ such that $1\le 2r-1\le n_l,1\le l\le s$, then 
$c'_{2r-1,l}x$ is {\em uniquely} expressible  
as $c'_{2r_0-1,l_0}y$ with $y$ a monomial 
in the generators $c_{2p-1,2q-1,i,j}^{\prime\prime}$ and $(r_0,l_0)$ is the {\em least} in lexicographical order
among such expressions. 
This shows that, as $R$-submodules of $\mathcal R/\mathcal J$,
\[\sum (\mathcal R_0/\mathcal J_0)c'_{2p-1,i}~\cong \bigoplus_{t\ge 0} Ru_m\otimes H^{4t+2}(\mathbb CG(\nu); R).\]   
It follows that 
$\mathcal R/\mathcal J\cong \mathcal R_0/\mathcal J_0\oplus (\bigoplus_{t\ge 0}(Ru_m\otimes H^{4t+2}(\mathbb CG(\nu);R)))$ is an isomorphism.  Thus, in view 
of Equation \ref{m-even-fixtheta-star}, we have the $R$-algebra isomorphism
$\mathcal R/\mathcal J\cong \fix(\theta^*)\cong H^*(P(m,\nu);R)$ 
defined by $\eta.$ 
\end{proof}

\section{$K$-theory of $P(m,\nu)$}

Our goal is to determine the ring structure of the complex $K$-theoretic ring 
$K^*(P(m,\nu))$ for any $m\ge 1$ and $\nu=(n_1,\ldots,n_s), s\ge 2.$ However, 
we have been greeted with limited success in our efforts. However, our results 
are complete up to a `finite ambiguity' in a sense that will be made precise 
later.

We now proceed to describe our results on (a) the additive structure of 
$K^*(P(m,\nu)),$  
(b) the ring structure of $K^*(P(m,\nu))\otimes \mathbb Q$, and, (c) a subring of $K^0(P(m,\nu))$ generated by the classes of certain canonical (complex) vector bundles on $P(m,\nu).$  

\subsection{ The additive structure of $K^*(P(m,\nu))$} 

Recall, from \cite{atiyah-hirzebruch},
the Atiyah-Hirzebruch spectral sequence $(E_r^{p,q},d_r)$, which abuts to $K^*(B)$ for a finite CW complex $B$.   The  
$E_2$-page is defined as $E_2^{p,q}=H^p(B;K^q(\{*\}))$ where $\{*\}$ denotes a one-point space.  The differential  
is of bidegree $d_r: E^{p,q}_r\to E_r^{p+r,q+1-r} $ has bidegree $(r,1-r).$
As such $d_r=0$ if $r$ is even, since $K^q(\{*\})=0$ for $q$ odd.

Let $\omega$ be a complex vector bundle over $B$. Recall that 
the Chern character $\textrm{ch}(\omega)\in H^*(B;\mathbb Q)$ is defined as 
$\textrm{ch}(\omega)=\sum_{k\ge 0} (\sum_{1\le j\le r} x_i^k/k!)$ where $x_j, 1\le j\le r$ are the `Chern roots' of $\omega.$  Thus, the $x_j$ are defined by a formal factorization of the 
Chern polynomial: 
\[ c(\omega,t)=\sum_{0\le p\le r}c_p(\omega)t^p=\prod_{1\le j\le r}(1+x_jt).\]  
Atiyah and Hirzebruch \cite{atiyah-hirzebruch} proved that the Chern character $\mathrm{ch}:K^0(B)\otimes \mathbb Q\to H^{\textrm{ev}}(B;\mathbb Q)=\bigoplus_{q\ge 0}H^{2q}(B;\mathbb Q)$ defined by $[\omega]\mapsto \textrm{ch}(\omega)$ is an isomorphism of rings.  In particular, the rank of $K^0(B)$ equals that of $H^{\textrm{ev}}(B;\mathbb Z)$.  It follows that if $H^*(B;\mathbb Z)$ has no 
$p$-torsion for a prime $p$, neither does $K^0(X).$   The same result applied 
to the suspension $S(B)$ of $B$ yields an isomorphism 
$K^1(B)\otimes \mathbb Q\to H^{\textrm{odd}}(B;\mathbb Q)=\bigoplus_{q\ge 1}H^{2q-1}(B;\mathbb Q).$

Taking $B$ to be $P(m,\nu)$ we obtain the following.  
Denote by $\xi_\mathbb C$ the 
complexification of the Hopf line bundle $\xi$ over $\mathbb RP^m.$  
Clearly $\xi_\mathbb C\otimes \xi_\mathbb C\cong \epsilon_\mathbb C$ and so $([\xi_\mathbb C]-1)^2
=-2([\xi_\mathbb C]-1).$  Adams \cite{adams} showed that $K^0(\mathbb RP^m)\cong \mathbb Z[y]/\langle y^2+2y, y^{r+1}\rangle=\mathbb Z\oplus \mathbb Z_{2^{r}}y$ where $y=[\xi_\mathbb C]-1$ and  $r=\lfloor m/2\rfloor.$  We shall denote by the same symbol $\xi_\mathbb C$ the pull-back $p^!(\xi_\mathbb C)$
on $P(m,\nu)$.

\begin{theorem}\label{k-groups-pmnu}
Let $m\ge 1$, $\nu=(n_1,\ldots, n_s), s\ge 2.$ Let $b_e,b_o, b'_e,b'_o$ be as in 
Theorem \ref{betti-torsion-coeffecients}.
Then:\\ (i)  $K^0(P(m,\nu))\cong \mathbb Z^{b_e}\oplus A_0$ where $A_0$ is a finite abelian group 
of order $2^k$ for some $k$,  with $0\le k\le  b_e'.$  The group $A_0$ contains a summand $\mathbb Z_{2^{\lfloor m/2\rfloor}}$, generated by $y=[\xi_\mathbb C]-1.$\\
(ii) $K^1(P(m,\nu))\cong \mathbb Z^{b_o}\oplus A_1$, where $A_1$ is a finite abelian group 
of order $2^t$ for some $0\le t\le b_o'$. \\ 
In particular, $K^1(P(m,\nu))$ is a torsion group if $m$ is even,
and  $K^0(P(1,\nu))$ is a torsion-free group 
\end{theorem}

\begin{proof}
Since $\textrm{ch}:K^*(P(m,\nu))\otimes \bq \to  H^*(P(m,\nu);\mathbb Q)$ is an 
isomorphism and since $K^*(P(m,\nu))$ is finitely generated group, it follows 
that the ranks of $K^0(P(m,\nu)), K^1(P(m,\nu))$ are the same as those of $H^{\textrm{ev}}(P(m,\nu);\mathbb Z), H^{\textrm{odd}}(P(m,\nu);\mathbb Z)$ 
respectively.  This shows that the ranks of $K^0(P(m,\nu)),K^1(P(m,\nu))$ as 
asserted.

 Since the Atiyah-Hirzebruch spectral sequence converges to $K^*(P(m,\nu)),$ 
$K^0(P(m,\nu))$ has a filtration so that the associated graded 
group is isomorphic to $\bigoplus_{p\ge 0} E_\infty^{p,-p}.$
Since the torsion subgroup of $\bigoplus_{p\ge 0} E^{2p,-2p}_2=H^{\textrm{ev}}(P(m,\nu);\mathbb Z)$ 
is a $2$-group of order $2^{b'_e}$, we obtain that 
the torsion subgroup of $K^0(P(m,\nu))$ is a $2$-group 
whose orders are bounded above by $2^{b'_e}$.
Similar argument yields that the torsion subgroup of $K^1(P(m,\nu))$ is a $2$-group 
of order at most $2^{b'_o}.$

To see that $K^0(P(m,\nu))$ contains a summand isomorphic to $\mathbb Z_{2^{\lfloor m/2\rfloor}},$ we note that 
$p:P(m,\nu)\to \mathbb RP^m$ admits a cross-section, namely $s([v])=[v,\mathbf E]$ where 
$\mathbf E=(\mathbb C^{n_1},\ldots,\mathbb C^{n_s})$ in which the standard basis vectors $e_{n_1+\cdots+n_{j-1}+1},\ldots, e_{n_1+\cdots+n_j}$ span the $j$-th coordinate   
$\mathbb C^{n_j}.$ 
It follows that the composition $s^!\circ p^!$ is the identity map of $K^0(\mathbb RP^m)$.  Hence 
$p^!$ is a monomorphism whose image is a summand of $K^0(P(m,\nu)$.  
By the work of Adams recalled above $\widetilde{K}^0(\mathbb RP^m)=\mathbb Z_{2^{\lfloor m/2\rfloor}}y$.  This completes the proof. 
\end{proof}

We pause for an example.
\begin{example}
(i) In the case of classical Dold manifolds, $\nu=(1,n-1)$ and $P(m,\nu)=P(m,n-1).$
The group $K^*(P(m,n-1))$ had been completely determined by Fujii 
\cite[Theorem 3.14]{fujii-66}.  The table below summarises Fujii's results (in our notations) yielding  
the rank and the order of the torsion groups of $K^0(P(m,n-1)), K^1(P(m,n-1))$.  Our results are consistent with 
Fujii's work but we have not been able to determine the groups $A_0,A_1$.   
\[
\begin{array}{|c|c||c|c||c|c|}
\hline
\multicolumn{2}{|c||}{P(m,n-1)} & \multicolumn{2}{c||}{K^0} & \multicolumn{2}{c|}{K^1} \\ 
\hline
m & n-1 & b_e & o(A_0) & b_o & o(A_1) \\ 
\hline \hline
2r & 2t & 2t+1 & 2^{r} & 0 &  0\\
\hline
2r+1 & 2t & t+1 & 2^{r} & t+1 & 0 \\
\hline
2r & 2t+1 & 2t +2 & 2^r & 0 & 2^{r}\\
\hline
2r+1& 2t+1 & t+1 & 2^r & t+1 & 2^{r+1} \\
\hline
\end{array}
\]

\begin{center}
    Table 2: The values of $b_e,b_o,o(A_0),o(A_1)$ for $P(m,n-1)$.
\end{center}

(ii) When $\mathbb CG(\nu)$ is the complete flag manifold, that is, when $\nu=(1,\ldots, 1)$,
$\nu_o=n$. We leave out the trivial case when $n=1$ and assume $n\ge 2.$  Now the values of $b_e,b_e',b_o,b_o'$ for $P(m,(1,\ldots,1))$  in  Theorem \ref{k-groups-pmnu} are described in the following table.
\[
\begin{array}{|c||c|c||c|c|}
\hline
P(m,\nu)&\multicolumn{2}{c||}{K^0}&\multicolumn{2}{c|}{K^1}\\
\hline
m &b_e&b_e'& b_o& b_o'\\
\hline\hline
2r &n!/2&rn!/2& n!/2  &rn!/2\\
\hline
2r+1 &n!&rn!/2& 0  & (r+1)n!/2\\
\hline
\end{array}
\]

\begin{center}
    Table 3: The values of $b_e,b_e',b_o,b_o'$ for $P(m,(1,\ldots,1))$.
\end{center}

\end{example}

\subsection{The ring structure of $K^0(P(m,\nu))$.}
Recall that $\gamma_j, 1\le j\le s,$ denotes the canonical complex $n_j$-plane 
bundle over $\mathbb CG(\nu)$. It is a $\sigma$-conjugate bundle, so we obtain 
the real vector bundle $\hat{\gamma}_j$ over $P(m,\nu).$  Denote by $\hat{\gamma}_{j\mathbb C}$ the complexification $\hat{\gamma}_j\otimes_\mathbb R\mathbb C.$ 

Let $f: \mathbb CG(\nu)\to P(m,\nu)$ be the inclusion of the fibre over the base point $[e_1]\in \mathbb RP^m$. Thus $f(\mathbf L)=[e_1,\mathbf L]$. Since $f^!(\hat\gamma_j)
=\gamma_j$ as {\em real} vector bundles, it follows that $f^!(\hat{\gamma}_{j\mathbb C})=\gamma_j\otimes \mathbb C\cong \gamma_j\oplus \bar\gamma_j$ as complex vector bundles.   Note that complexification commutes with Whitney sums, tensor products, 
and exterior products.   
So, we have $f^!(\Lambda^k(\hat{\gamma}_{j\mathbb C}))=\Lambda^k(\gamma_j\oplus \bar\gamma_j)\cong \bigoplus_{p+q=k} \Lambda^p(\gamma_j)\otimes \Lambda^q(\bar\gamma_j)
\cong \bigoplus_{p+q=k}\Lambda^p(\gamma_j)\otimes \overline{\Lambda^q(\gamma_j)}$ for all $k\ge 0, 1\le j\le s.$

We recall the following description of the complex $K$-ring of $\mathbb CG(\nu)$. See \cite[Theorem 3.6, Chapter-IV]{karoubi}.

\begin{theorem}\label{k-theory-flag-manifold} The ring 
    $K^0(\mathbb CG(\nu))$ is isomorphic to the polynomial ring  
    generated by $\lambda_{p,j},~1\le p\le n_j, ~1\le j\le s$, modulo the 
    ideal generated by the set $\{h_p-{n\choose p}\mid 1\le p\le n\},$ where  $h_p=h_p(\lambda_{q,j})$ is the coefficient of $t^p$ in the following polynomial 
    in the variable $t$:
    \[ \sum_{0\le p\le n}h_pt^p = \prod_{1\le j\le s}(\sum_{0\le r\le n_j}\lambda_{r,j}t^r).\] 
    Here $h_0=1, \lambda_{0,j}=1~\forall j.$ The isomorphism is defined by 
    sending $\lambda_{p,j}$ to $[\Lambda^p(\gamma_j)].$
    Moreover, $K^1(\mathbb CG(\nu))=0.$ \hfill $\Box$
\end{theorem}

We shall refer to the generators $\lambda_{p,j}$ as the {\em canonical generators} 
of $K^0(\mathbb CG(\nu)).$
\begin{remark}\label{all-bundles-conjugate}
For any compact connected Lie group $G$ and a closed subgroup $H\subset G,$  
the $\alpha$-construction yields a $\lambda$-ring homomorphism $RH\to K^0(G/H)$ 
defined by that sends $[V]\in RH$ to the class of the associated vector bundle with projection
$G\times_HV\to G/H$. (See \cite{atiyah-hirzebruch}, \cite{husemoller}.)
The bundle $\gamma_j$ is associated to the representation $\lambda_{1,j}: S(U(\nu)):=S(U(n_1)\times \cdots \times U(n_s))\to U(n_j)$, the projection to the $j$th coordinate, via the so-called `$\alpha$-construction'. Taking the $p$th exterior 
power, we obtain that $\Lambda^p(\gamma_j)$ is associated to $\Lambda^p(\lambda_{1,j})=\lambda_{p,j}$ in the representation ring $RS(U(\nu)).$  
Using the known description of the representation of the unitary group (see \cite[Chapter 14]{husemoller}), 
it is readily seen that $RS(U(\nu))=\mathbb Z[\lambda_{p,j}, \lambda_{n_j ,j}^{- 1}\mid 1\le j\le n_j]/\langle \prod_{1\le j\le s} \lambda_{n_j,j}-1\rangle. $ Note that  
$\lambda_{n_j,j}^{-1}=\prod_{i\ne j} \lambda_{n_i,i}$ in $RS(U(\nu)).$ 
Denote by $\lambda_1$ the standard representation of $SU(n)$ on $\mathbb C^n$. 
We denote by $\rho([V])\in RS(U(\nu))$  the restriction of the class $[V]$ of a $S(U(\nu))$ representation $V.$ Thus $\rho(\lambda_1)=\sum_{1\le j\le s}\lambda_{1,j}.$
Set $\Lambda_t([V]):=\sum_{1\le j\le r}[\Lambda^q(V)]t^q.$
Then $\Lambda_t(\lambda_1)=\prod_{1\le j\le s}\Lambda_t(\lambda_{1,j}).$

If a complex representation of $S(U(\nu))$ arises as the extension of a representation of $SU(n),$ then the associated vector bundle is trivial. It follows that 
$\rho(h_p)={n\choose p}$. So $\rho$ defines 
a ring homomorphism $\bar\rho: RS(U(\nu))/\mathcal I\to K^0(\mathbb CG(\nu))$ 
where $\mathcal I\subset RS(U(\nu))$ is the ideal in $\langle h_p-{n\choose p}\mid 1\le p\le n-1 \rangle.$ 
The above theorem is equivalent to the assertion that $\bar\rho$ is an isomorphism.

\end{remark}

As a corollary, we obtain the following.

\begin{lemma}\label{k-flag-generatedbyconjuagatebundles}
$K(\mathbb CG(\nu))$ is generated as an abelian group by the classes of 
$\sigma$-conjugate complex vector bundles.
\end{lemma}
\begin{proof}
  We need only observe that if $\xi$ and $\eta$ are $\sigma$-conjugate complex vector bundles so are $\xi\otimes\eta,\Lambda^q(\xi), q\ge 0, $ and $\xi\otimes \eta.$ 
  (See \cite[Example 2.2(iv), Lemma 2.3(ii)]{nath-sankaran}.)
  Since the canonical bundles $\gamma_j,1\le j\le s,$ are all $\sigma$-conjugate bundles,
  the lemma follows from the above theorem. 
\end{proof}

Since $\pi=\pi\circ \theta,$ we have $\theta^!\circ \pi^!=(\pi \circ \theta)^!=\pi^!.$ 
So, for any complex vector bundle $\omega$ 
over $P(m,\nu),$
$\pi^!(\omega)$ is a complex vector bundle over $\mathbb S^m\times \mathbb CG(\nu)$ such that $\theta^!(\pi^!(\omega))=\pi^!(\omega)$.  This implies that the image of $\pi^!:K^0(P(m,\nu))\to K^0(\mathbb S^m\times \mathbb CG(\nu))$ is contained in the fixed subring $\fix(\theta^!)$.

Since $K^1(\mathbb CG(\nu))=0$, we have isomorphisms
given by the exterior tensor products of vector bundles. 
See \cite[Corollary 2.7.15.]{atiyah}, \cite[Prop. 3.24, Chapter IV]{karoubi}.
\begin{equation} K^0(\mathbb S^m\times \mathbb CG(\nu))=K^0(\mathbb S^m)\otimes K^0(\mathbb CG(\nu)),
\end{equation}
and,
\begin{equation} K^1(\mathbb S^m\times \mathbb CG(\nu))=K^1(\mathbb S^m)\otimes K^0(\mathbb CG(\nu)).
\end{equation}

Also, if $m$ is even $K^1(\mathbb S^m)=0$ and if $m$ is odd, $\widetilde{K}(\mathbb S^m)=0$.

We shall identify $ y\in K^0(\mathbb CG(\nu))$ with $pr_2^!(y)=1\otimes y\in K^0(\mathbb S^m\times \mathbb CG(\nu))$ and $x\in K^0(\mathbb S^m)$ with 
$x\otimes 1=pr_1^!(x)$ where $pr_i$ denotes the projection to the $i$th factor of 
$\mathbb S^m\times \mathbb CG(\nu)$. Thus $x\otimes y$ is identified with $xy\in K^0(\mathbb S^m\times \mathbb CG(\nu)).$   A similar notation 
holds for the $K^0(\mathbb S^m\times \mathbb CG(\nu))$-module 
$K^1(\mathbb S^m\times \mathbb CG(\nu)).$

Let $m=2r$ be even. By 
Theorem \ref{half-spin-bundles on even spheres}, the element $[\xi^+]-2^{r-1}=2^{r-1}-[\xi^-]$ generates $\ker(\mathrm{rank}:K^0(\mathbb S^m)\to\mathbb Z)\cong \mathbb Z$.

\begin{definition}\label{delta-0}
Let $m\ge 0$.  We set \[\delta_m:=\begin{cases}
    [\xi^+]-2^{r-1}=2^{r-1}-[\xi^-]& ~\textrm{if~} m~\equiv 0\pmod 2\\
    0& ~\textrm{if~} m~\equiv 1\pmod 2.\\
\end{cases}
\]
  \end{definition}

We have $\theta^!(\delta_m)=\alpha^!(\delta_m)=-\delta_m$. (This follows from the naturality of the Chern character.)

For any $x\in K^0(\mathbb CG(\nu))$, $\theta^!(x)=\theta^!(1\otimes x)=\sigma^!(x)=\bar x.$

\begin{lemma} \label{fix-k-ring-pmnu}  Let $m\ge 1$ be arbitrary. 
Let $\mathcal K$ be the subring of $K^0(\mathbb S^m\times \mathbb CG(\nu))$ generated by the elements (i)  $x+\bar x$ and, 
    (ii) $\delta_m(x-\bar x)$ when $m$ is even, 
    as $x$ varies in $K^0(\mathbb CG(\nu)).$
Then the subring $\fix(\theta^!)\subset K^0(\mathbb S^m\times \mathbb CG(\nu))$ contains $\mathcal K.$ 
    Moreover, if $z\in \fix(\theta^!)$, then $2z\in  \mathcal K.$ Thus 
    \[   2\fix(\theta^!)\subset \mathcal K\subset \fix(\theta^!).  \]
\end{lemma}
\begin{proof}
    It is evident that, for any $x\in K^0(\mathbb CG(\nu))$, $x+\bar x\in \fix(\theta^!)$.  
    
    When $m$ is even, 
    $\theta^!(\delta_m(x-\bar x))=\alpha^!(\delta_m)(\bar x-x)=-\delta_m(\bar x-x)=\delta_m(x-\bar x)$ and so $\delta_m(x-\bar x)\in \fix(\theta^!).$   
    
    This shows that $\mathcal K\subset \fix(\theta^!)$ for any $m\ge 1.$

    Let $z\in \fix(\theta^!)$ where we assume that $\textrm{rank}(z)=0.$ 
   Write $z=z_0+\delta_mz_1$ where $z_i\in \widetilde{K}^0(\mathbb CG(\nu))$ (recall that  
   $\delta_m=0$ if $m$ is odd). 
   Then $z=\theta^!(z)=\bar z_0-\delta_mz_1$ and so $2z= z+\theta^!(z)
 =z_0+\bar z_0+\delta_m(z_1-\bar z_1)\in \mathcal K.$  
   \end{proof}

    \begin{remark} (i). If $y\in K^0(\mathbb CG(\nu)),$ then, $y^2+\bar y^2, (y+\bar y)^2\in \mathcal K.$  Therefore 
    $2y\bar y=(y+\bar y)^2-y^2-\bar y^2\in \mathcal K. $\\
    (ii)
Since $\bar\lambda_{n_j,j}=\lambda_{n_j,j}^{-1}=\prod_{i\ne j}\lambda_{n_i,i},$ we may 
    write any $x\in K^0(\mathbb CG(\nu))$ as $x=P(\lambda_{p,j})$ as a polynomial in the canonical 
    generators $\lambda_{p,j}, 1\le p\le n_j, 1\le j\le s,$ 
    of $K^0(\mathbb CG(\nu)).$   In particular, 
    the exponents of $\lambda_{n_j,j}, 1\le j\le s, $ that occur in $P(\lambda_{p,j})$ is non-negative for each $j.$   Furthermore,  
    we may assume that, in each monomial, at least one of the $\lambda_{n_j,j}$ does not occur.  This is possible since $\prod_{1\le j\le s} \lambda_{n_j,j}=1$. 
   \end{remark}

Recall from \S\ref{bott-bundle} that, when $m\equiv 0\pmod 2,$  
the complex vector bundles $\xi^+, \xi^-$ over $\mathbb S^m$ are associated to the half-spin representations, and the bundle $\eta^-$ is defined as the pull-back $\alpha^!(\xi^+)$. 
Also we obtained bundle maps $\tilde \alpha^-:\eta^-\to \xi^+$ and $\tilde \alpha^+:\xi^+:\xi^+\to \eta^-$ both covering $\alpha$ such that $\tilde\alpha^-\circ\tilde\alpha^+$ and $\tilde\alpha^+\circ \tilde\alpha^-$ are the identity isomorphisms $id_{\xi^+}$ and $id_{\eta^-}$ respectively.

Let $\tilde \sigma: E(\omega)\to E(\omega)$ be a $\sigma$-conjugate bundle involution that covers $\sigma:\mathbb CG(\nu)\to \mathbb CG(\nu).$
We regard it as a complex linear isomorphism 
$\omega\to \bar \omega$ covering $\sigma.$
We have a complex vector bundle isomorphism $\tilde \theta^+: E(\xi^+\otimes \omega)\to E(\eta^-\otimes \bar\omega)$ 
that covers $\theta.$  Explicitly, $\tilde\theta^+ (e^+\otimes u)=\tilde \alpha^+(e^+)\otimes \tilde \sigma(u)$ and $\tilde\theta^- (e^-\otimes v)=\tilde \alpha^-(e^-)\otimes \tilde \sigma(v).$  The bundle isomorphism $\tilde\theta^-:E(\eta^-\otimes \bar\omega)\to E(\xi^+\otimes \omega)$ covering $\theta$ is defined analogously. 
We have 
\begin{equation}\label{theta+theta-}
\tilde\theta^+\circ \tilde\theta^-=id_{\eta^-\otimes \bar\omega}, \textrm{~and~} \tilde \theta^-\circ \tilde \theta^+=id_{\xi^+\otimes \omega}.
\end{equation}
   
    Set $\tilde\xi(\omega):=\xi^+\otimes \omega\oplus \eta^-\otimes \bar\omega$ for any $\sigma$-conjugate vector bundle $\omega.$

\begin{lemma}\label{spin-omega} 
    With the above notations, we have:\\ (i) For any $m\ge 1$, we have $\pi^!(\hat\omega_\mathbb C)\cong \omega\oplus \bar\omega$.\\
    (ii) Let $m\equiv 0\pmod 2.$  Then 
    there exists a complex vector bundle $\xi^0(\omega)$ over $P(m,\nu)$ such that $\pi^!(\xi^0(\omega))\cong \tilde\xi(\omega)$.
\end{lemma}
\begin{proof}  We have identified $\omega$ with the bundle $pr_2^!(\omega)$ on $\mathbb S^m\times \mathbb CG(\nu)$ in the statement and will do so in the proof as well. 

(i) We shall prove the stronger statement that 
$\pi^!(\hat\omega)\cong \omega_\mathbb R,$ the real vector bundle underlying $\omega$.
Let $\hat\sigma$ be the $\sigma$-conjugate vector bundle morphism of $\omega$ that 
covers $\sigma:\mathbb CG(\nu)\to \mathbb CG(\nu).$
Recall that the total space $E(\hat\omega)=P(\mathbb S^m, E(\omega))$ of $\hat\omega$ is obtained as the quotient of 
$\mathbb S^m \times E(\omega)$ where $(v,e)$ is identified to $(-v, \hat\sigma(e)).$  
The vector bundle projection $p_{\hat\omega}:E(\hat\omega)\to P(m,\nu)$ maps $[v, e]$ to $[v,p_\omega(e)].$  We have a commuting diagram 
where the horizontal maps are quotient maps and the vertical maps are bundle projections:
\begin{equation}
\begin{tikzcd}[ampersand replacement=\&]
	{\mathbb S^m\times E(\omega)} \& {E(\hat\omega)} \\
	{\mathbb S^m\times \mathbb CG(\nu)} \& {P(m,\nu)}
	\arrow["{\hat \pi}", from=1-1, to=1-2]
	\arrow["{id\times p_{\omega}}"', shift right=2, from=1-1, to=2-1]
	\arrow["{p_{\hat\omega}}", from=1-2, to=2-2]
	\arrow["\pi", from=2-1, to=2-2]
\end{tikzcd}   
\end{equation}

Since $\hat\pi$ is an $\mathbb R$-linear isomorphism on each fibre $v\times E_x(\omega)$,  it follows that $\pi^!(\hat \omega)\cong\omega_\mathbb R.$ 

(ii) We have a bundle involution $\tilde \theta:E(\tilde\xi(\omega))\to E(\tilde\xi(\omega))$ that covers $\theta$ defined as:
\[\tilde \theta (e^+\otimes u,e^-\otimes v)=(\tilde\theta^-(e^-\otimes v), \tilde\theta^+(e^+\otimes u)),
\]
for $e^+\otimes u\in E_b(\xi^+\otimes \omega), e^-\otimes v\in E_b(\eta^-\otimes \bar\omega)$, 
$b\in \mathbb S^m\times \mathbb CG(\nu)$.  Note that $\tilde\theta$ is fixed point free. So 
we obtain a vector bundle $\xi^0(\omega)$ over $P(m,\nu)$, whose total space is $E(\tilde\xi(\omega))/\langle \tilde \theta\rangle$ and we have the following commuting 
diagram in which the double covering projection $\tilde \pi$  
is a bundle map:
\begin{equation}
\begin{tikzcd}[ampersand replacement=\&]
	{E(\tilde\xi(\omega))} \& {E(\xi^0(\omega))} \\
	{\mathbb S^m\times \mathbb CG(\nu)} \& {P(m,\nu)}
	\arrow["{\tilde \pi}", from=1-1, to=1-2]
	\arrow["{\tilde p}"', from=1-1, to=2-1]
	\arrow["\pi", from=2-1, to=2-2]
	\arrow["{p}", from=1-2, to=2-2]
\end{tikzcd}
\end{equation}
Therefore $\pi^!(\xi^0(\omega))=\tilde\xi(\omega).$ 
\end{proof}

It is not clear whether there is a {\em unique} bundle (up to isomorphism) $\xi^0(\omega)$ such that $\pi^!(\xi^0(\omega))\cong \tilde\xi(\omega)$. (However, $\xi_\mathbb C\otimes \xi^0(\omega)\cong \xi^0(\omega)$; see 
Lemma \ref{multiplication-k}(ii) below.) 
In the sequel, $\xi^0(\omega)$ will always denote the bundle as constructed in the above 
proof. 
The following lemma implies, that the class of any other such bundle differs from $[\xi^0(\omega)]$
by a torsion element in $K^0(P(m,\nu))$.

\begin{lemma} \label{k-subring-image}  Let $m\ge 1$ be arbitrary.  
With notations as in Lemma \ref{fix-k-ring-pmnu},\\ $\ker{\pi^!:K^0(P(m,\nu))\to K^0(\mathbb S^m\times \mathbb CG(\nu))}$ is precisely the torsion subgroup of $K^0(P(m,\nu))$ and $Im(\pi^!)$ contains $\mathcal K.$
\end{lemma}
\begin{proof}  
By the naturality of the Chern character, we have the following commutative diagram:
\begin{equation}
    \begin{tikzcd}[ampersand replacement=\&]
	{K^0(P(m,\nu))} \& {K^0(\mathbb S^m\times \mathbb CG(\nu))} \\
	{H^*(P(m,\nu);\mathbb Q)} \& {H^*(\mathbb S^m\times \mathbb CG(\nu);\mathbb Q)}
	\arrow["{\pi^!}", from=1-1, to=1-2]
	\arrow["{\textrm{ch}}"', from=1-1, to=2-1]
	\arrow["{\pi^*}", from=2-1, to=2-2]
	\arrow["{\textrm{ch}}", from=1-2, to=2-2]
\end{tikzcd}
\end{equation}

Since $\textrm{ch}:K^0(\mathbb S^m\times \mathbb CG(\nu))\to H^*(\mathbb S^m\times \mathbb CG(\nu);\mathbb Q)$ is a monomorphism, and since $\pi^*:H^*(P(m,\nu);\mathbb Q)
\to H^*(\mathbb S^m\times \mathbb CG(\nu);\mathbb Q)$ is a monomorphism,
$\ker \pi^!$ equals the kernel of 
$\textrm{ch}: K^0(P(m,\nu))\to H^*(P(m,\nu);\mathbb Q),$ which is precisely 
the torsion subgroup of $K^0(P(m,\nu)).$ 

It remains to show that $\mathcal K\subset Im(\pi^!).$  Recall that, by definition,
$\delta_m=0$ when $m$ is odd.
Since $\mathcal K$ is generated as a 
subring by $x+\bar x, \delta_m(x-\bar x)$ where $x$ varies in $ K^0(\mathbb CG(\nu)),$ we 
need only to show that $x+\bar x$ and $\delta_m(x-\bar x)$ are in $Im(\pi^!).$

By Lemma \ref{k-flag-generatedbyconjuagatebundles}, any $x\in \widetilde{K}(\mathbb CG(\nu))$ 
may be expressed $[\omega]-d$ where $\omega$ is a $\sigma$-conjugate complex vector bundle and 
$d=\textrm{rank}(\omega).$ Thus $\theta^!(x)=\sigma^!(x)=[\bar\omega]-d=\bar x.$ Consider the complex vector bundle $\hat\omega_\mathbb C=\hat\omega\otimes_\mathbb R \mathbb C,$ the complexification of $\hat\omega.$   We have $\pi^!([\hat\omega_\mathbb C]-2d)=[\omega]+[\bar\omega]-2d=x+\bar x.$  Thus $x+\bar x\in Im(\pi^!)$. This proves 
that $\mathcal K\subset Im(\pi^!)$ when $m$ is odd.

Let $m\equiv 0\pmod 2.$
Recall from Remark \ref{eta-xi} that $[\xi^-]=[\eta^-]$ in $K(\mathbb S^m)$ and so 
$\delta_m=[\xi^+]-2^{r-1}=-(
[\eta^-]-2^{r-1})$ 
in 
$K^0(\mathbb S^m\times \mathbb CG(\nu)).$  Therefore 
$[\tilde\xi(\omega)]=[\xi^+][\omega]+[\eta^-][\bar\omega]=
\delta_m([\omega]- [\bar\omega])
+2^{r-1}([\omega]+[\bar \omega])$.  By Lemma \ref{spin-omega}, 
$[\tilde\xi(\omega)]\in Im(\pi^!)$.  It follows that $\delta_m(x-\bar x)=\delta_m\cdot 
([\omega]-[\bar\omega])
\in Im(\pi^!)$ since 
$[\omega]+[\bar\omega]\in Im(\pi^!).$  Hence $\mathcal K\subset Im(\pi^!).$
\end{proof}

Recall that $\xi_\mathbb C$ denotes the complexification of the Hopf line bundle
$\xi$ over $\mathbb RP^m.$  We shall denote by the same symbol $\xi$ (resp. $\xi_\mathbb C$) the bundle $p^!(\xi)$ (resp.$p^!(\xi_\mathbb C)$) over 
$P(m,\nu)$.

\begin{definition} Let $m\ge 1$. 
 With notations as above, 
define $\mathcal K^0\subset K^0(P(m,\nu))$ to be the subring generated by the following 
elements:\\
(i) $[\xi_\mathbb C],$ 
$[\hat\omega_\mathbb C]$, and, (ii) when $m$ is even,  $[\xi^0(\omega)],$\\
where $\omega$ varies over $\sigma$-conjugate complex vector bundles over $\mathbb CG(\nu)$.
\end{definition}

We have the following proposition which gives partial information on the multiplicative structure of $K^0(P(m,\nu)).$

\begin{proposition} \label{multiplication-k} Let $m\ge 1$. The following relations hold in $K^0(P(m,\nu))$ for any $\sigma$-conjugate vector bundles $\omega_1,\omega_2$ and $\omega$ over $\mathbb CG(\nu)$.\\
    (i)
    $([\xi_\mathbb C]-1)[\hat\omega_\mathbb C]=0$, and, The (additive) order of $([\xi_\mathbb C]-1)$ equals $2^{\lfloor m/2\rfloor}$.\\ 
    Furthermore, when $m$ is even, we have\\
    (ii) $([\xi_\mathbb C]-1)[\xi^0(\omega)]=0$, $[\xi^0(\omega_1)]+[\xi^0(\omega_2)]
    =[\xi^0(\omega_1\oplus \omega_2)]$;\\
    (iii)  
    $[\xi^0(\omega)]+[\xi^0(\bar \omega)]\equiv 2^r[\hat\omega_\mathbb C]$ modulo torsion; \\
    (iv) 
    $[\xi^0(\omega_1)][\xi^0(\omega_2)]\equiv 2^r[\xi^0(\omega_1\otimes \omega_2)]-2^{2r-2}([\omega_1]-[\bar\omega_1])([\omega_2]-[\bar\omega_2])$ modulo torsion.    
\end{proposition}
\begin{proof}
(i) 
The assertion that $([\xi_\mathbb C]-1)[\hat\omega_\mathbb C]=0$ 
holds since $\hat\omega\otimes \xi\cong \hat \omega$. (See \cite[Lemma 2.7]{nath-sankaran}.) 
That $2^{\lfloor m/2\rfloor}$ equals the additive order of $([\xi_\mathbb C]-1)=0$ follows readily 
from the work of Adams \cite{adams} in view of the fact that the projection $p:P(m,\nu)\to \mathbb RP^m$ induces a monomorphism 
$p^!:K^0(\mathbb RP^m)\to K^0(P(m,\nu))$ since $p$ admits a cross-section. 

(ii)   The second assertion of (ii) is readily seen to be valid.
We shall construct a bundle isomorphism 
$\tilde \lambda:\epsilon_\mathbb C\otimes \tilde\xi(\omega)\to \tilde \xi(\omega)$ that 
covers $\theta:\mathbb S^m\times \mathbb CG(\nu)\to \mathbb S^m\times \mathbb CG(\nu)$
and such that it 
descends to yield a bundle isomorphism 
$\xi_\mathbb C\otimes \xi^0(\omega)\cong \xi^0(\omega)$ on $P(m,\nu)$. This readily 
implies the first assertion of (ii).

We have $E(\xi_\mathbb C)=\mathbb S^m\times \mathbb CG(\nu)\times \mathbb C/\langle \mu\rangle $ where 
$\rho:\epsilon_\mathbb C\to \epsilon_\mathbb C$ is the involutive bundle isomorphism that covers $\theta$ defined as $\rho(b;t)\:=(\theta(b);-t)~\forall b\in \mathbb S^m\times \mathbb CG(\nu), t\in \mathbb C.$ 
We shall denote by $[b;t]\in E_{[b]}(\xi _\mathbb C)$ the image of $(b;t)\in E_b(\epsilon_\mathbb C)$ under bundle map $\epsilon_\mathbb C\to \xi_\mathbb C$ covering the 
projection $\mathbb S^m\times \mathbb CG(\nu)\to P(m,\nu)$. Thus $[b;t]=[\theta(b);-t]$.

 The total space $E(\xi^0(\omega))$ of $\xi^0(\omega)$ was described 
in the course of the proof of Lemma \ref{spin-omega} as $E(\tilde \xi)/\langle \tilde\theta\rangle .$ We shall denote by $(b;x,y)\in E_b(\tilde \xi(\omega))$ the element $(x,y)\in E_b(\xi^+\otimes \omega)\oplus E_b(\eta^-\otimes \bar\omega)$ and 
by $[b;x,y]$ its image in $E_{[b]}(\xi^0(\omega))$ under the projection $E(\tilde \xi(\omega))\to E(\xi^0(\omega))$. 
Thus $[b;x,y]=[\theta(b); \tilde\theta^-(y), \tilde \theta^+(x)].$  

The total space of $\xi_\mathbb C\otimes \xi^0(\omega)$ has the following description:
The fibre over $[b]\in P(m,\nu)$ is the vector space $E_{[b]}(\xi_\mathbb C)\otimes E_{[b]}(\xi^0(\omega)).$  Let $t\in \mathbb C, x\in E_b(\xi^+\otimes \omega), y\in E_b(\eta^-\otimes \bar\omega)$ so that $[b;t]\in E_{[b]}(\xi_\mathbb C), [b;x,y]\in E_{[b]}(\xi^0(\omega)).$
Then the vector $[b;t]\otimes [b;x,y]\in E_{[b]}(\xi_\mathbb C\otimes \xi^0(\omega))=
E_{[b]}(\xi_\mathbb C)\otimes E_{[b]}(\xi^0(\omega))$  
will be denoted $[(b;t)\otimes (x,y)].$  Note that 
\begin{equation}\label{xic-xi0omega}
[(b;t)\otimes (x,y)]=[(\theta(b);-t)\otimes (\tilde \theta^{-}(y),\tilde \theta^+(x))].
\end{equation}

Consider the bundle map 
$\tilde \lambda:E(\epsilon_\mathbb C\otimes \tilde \xi(\omega))\to E(\tilde \xi(\omega))$ defined as follows: For $b\in \mathbb S^m\times \mathbb CG(\nu), t\in \mathbb C, x\in E_b(\xi^+\otimes \omega), y\in E_b(\eta^-\otimes \bar\omega)$, 
\begin{equation}
\tilde \lambda ((b;t)\otimes (b; x,y))=(b; -tx,ty).
\end{equation}
Then $\tilde \lambda$ is a complex vector bundle isomorphism that covers $\theta.$
Moreover, 
\[\begin{array}{rcl} 
\tilde \lambda \circ (\rho\otimes \tilde \theta)((b;t)\otimes (b;x,y))&=&
\tilde\lambda((\theta(b);-t)\otimes(\theta(b);\tilde\theta^-(y),\tilde\theta^+(x)))\\
&=&
(\theta(b);t\tilde\theta^-(y),-t\tilde\theta^+(x))\\
&=&\tilde\theta(b;-tx,ty)\\
&=&\tilde\theta(\tilde\lambda((b;t)\otimes(b;x,y)))
\end{array}
\]

Therefore $\tilde \lambda $ yields a well-defined bundle isomorphism 
$\lambda: \xi_\mathbb C\otimes \xi^0(\omega)\to \xi^0(\omega)$, 
covering the identity 
map of $P(m,\nu)$,
 defined as 
$[(b; t)\otimes (x,y)]\mapsto [b; -tx,ty].$
This shows that $\xi_\mathbb C\otimes \xi^0(\omega)\cong \xi^0(\omega)$.

(iii)  In view of the fact that $\ker(\pi^!)$ equals the torsion subgroup of $K^0(P(m,\nu)),$ 
it suffices to show that $\pi^!([\xi^0(\omega)]+[\xi^0(\bar \omega)])=2^r\pi^!([\hat\omega_\mathbb C]).$
Since $\pi^!(\hat\omega_\mathbb C)=([\omega]+[\bar\omega])$ by Lemma \ref{spin-omega} and since 
$\pi^!(\xi^0(\omega))=\tilde \xi(\omega)$, we need only to show that 
$[\tilde\xi(\omega)]+[\tilde\xi(\bar \omega)]=2^r([\omega]+[\bar\omega]).$
It is clear that $\tilde \xi(\omega)\oplus \tilde\xi(\bar\omega)
=( \xi^+\oplus \eta^-)\otimes (\omega\oplus \bar \omega)\cong 2^r(\omega\oplus \bar\omega)$
since we have a complex vector bundle isomorphism 
$\xi^+\oplus \eta^-\cong \xi^+\oplus \xi^-\cong 2^r\epsilon_\mathbb C$ on $\mathbb S^m$.

(iv) We have 
\[
\begin{array}{rcl}
\pi^!(\xi^0(\omega_1)\otimes \xi^0(\omega_2))
&=&\tilde\xi(\omega_1)\otimes \tilde\xi(\omega_2)\\
&=& (\tilde \xi^+\otimes \omega_1+\eta^{-}\otimes \bar \omega_1)\otimes (\tilde \xi^+\otimes  \omega_2+\eta^{-}\otimes\bar\omega_2)\\
&=&(\xi^+)^2\otimes(\omega_1\otimes \omega_2)
+(\eta^-)^2\otimes(\bar \omega_1\otimes\bar \omega_2)+\xi^+\eta^-\cdot(\omega_1\otimes \bar\omega_2\oplus \bar\omega_1\otimes \omega_2).
\end{array}
\]
Now by Theorem \ref{half-spin-bundles on even spheres} and Remark \ref{eta-xi}, 
$[\xi^+]^2=2^r[\xi^+]-2^{2r-2},[\eta^-]^2=2^r[\eta^-]-2^{2r-2}$ and $[\xi^+\eta^-]=2^{2r-2}.$  Substituting 
in the above equation and using $\tilde \xi(\omega)=\xi^+\otimes \omega\oplus \eta^-\otimes \bar \omega$ (with $\omega=\omega_1\otimes \omega_2$) we obtain: 
\[
\begin{array}{rcl}
\pi^!([\xi^0(\omega_1)\otimes \xi^0(\omega_2)])
&=&2^r([\tilde\xi(\omega_1\otimes \omega_2)])- 2^{2r-2}([\omega_1\otimes \omega_2]+[\bar\omega_1\otimes \bar\omega_2])\\
& &+2^{2r-2}([\omega_1\otimes\bar\omega_2]+[\bar\omega_1\otimes \omega_2])\\
&=&2^r \pi^!([\xi^0(\omega_1\otimes \omega_2)]-2^{2r-2}([\omega_1]-[\bar\omega_1])([\omega_2]-[\bar\omega_2]).\\ 
\end{array}
\]
Since $\ker(\pi^!:K(P(m,\nu))\to K(\mathbb S^m\times \mathbb CG(\nu)))$ consists 
only of torsion elements, our assertion follows. 
\end{proof}

\begin{theorem} \label{kpmnu-cofinite-ring}  Let $m\ge 1$ and let $r=\lfloor m/2 \rfloor$.
The homomorphism $\pi^!:K^0(P(m,\nu))\to K^0(\mathbb S^m\times \mathbb CG(\nu))$  defines a surjective ring homomorphism $\mathcal K^0\to \mathcal K,$ again denoted $\pi^!$,  
whose kernel equals the torsion ideal $\mathcal T^0\subset\mathcal K^0$.
The quotient group $K^0(P(m,\nu))/\mathcal K^0$ is a finite abelian $2$-group. 
\end{theorem}
\begin{proof} 
 We have 
$\pi^!([\hat\omega_\mathbb C])=[\omega]+[\bar\omega]\in \mathcal K$, and, when $m$ is even, $\pi^!([\xi^0(\omega)])=[\tilde\xi(\omega)]=[\xi^+][\omega]+[\eta^-][\omega]=[\xi^+]([\omega-\bar\omega])
+2^{r}[\bar\omega]=
\delta_m([\omega]-[\bar\omega])+2^{r-1}\pi^!([\hat\omega_\mathbb C])$.  
Therefore $\delta_m([\omega]-[\bar\omega])$ is in $\pi^!(\mathcal K^0).$ This shows that the homomorphism $\pi^!:\mathcal K^0\to \mathcal K$ is surjective. 

Since, by Lemma \ref{k-subring-image}, the kernel of $\pi^!:K^0(P(m,\nu))\to K^0(\mathbb S^m\times \mathbb CG(\nu))$ consists 
only of torsion elements, it follows that the same is true of 
$\mathcal K^0\to \mathcal K$.  By the same lemma, $\textrm{rank}(\mathcal K)=\textrm{rank}(K^0(P(m,\nu))),$ and since 
$\mathcal K^0\subset K^0(P(m,\nu))$, it follows from the surjectivity of $\mathcal K^0\to \mathcal K$ that $\textrm{rank}(\mathcal K^0)=\textrm{rank}(K^0(P(m,\nu)).$  Therefore $K^0(P(m,\nu))/\mathcal K^0$ 
is a torsion group.  Since $H^*(P(m,\nu);\mathbb Z)$ has no odd torsion, it follows (for example by using the Atiyah-Hirzebruch spectral sequence) that $K^0(P(m,\nu))$ has no odd torsion. 
This proves the last assertion of the theorem.
\end{proof} 

\begin{remark}
    It is clear that $\mathbb Z_{2^r}y$ is contained in the torsion subgroup of $K^0(P(m,\nu)).$ 
    We have not been able to determine the torsion part of $K^0(P(m,\nu)).$   This seems 
    to us to be a rather difficult, albeit an interesting problem. 
\end{remark}

\end{document}